\documentclass[preprint,11pt]{elsarticle}

\usepackage{amsmath,amssymb,amsthm,amsfonts,amsopn}
\usepackage{graphicx}
\usepackage{stmaryrd}
\usepackage{algorithmic}
\usepackage[margin=1in]{geometry}
\usepackage[hidelinks]{hyperref}
\usepackage{lineno}

\journal{Journal of Computational Physics}

\newtheorem{lemma}{Lemma}
\newtheorem{proposition}{Proposition}

\theoremstyle{remark}

\newcommand{\paren}[1]{\left(#1\right)}
\newcommand{\jump}[1]{\llbracket#1\rrbracket}

\newcommand{\at}[2]{\left. #1 \right|_{#2}}
\newcommand{\grad}[1]{\nabla #1}
\newcommand{\mb}[1]{\mathbf{#1}}
\newcommand{\mc}[1]{\mathcal{#1}}
\newcommand{\wh}[1]{\widehat{#1}}
\newcommand{\wt}[1]{\widetilde{#1}}
\newcommand{\bm}[1]{\boldsymbol{#1}}

\newcommand{\norm}[1]{\left\lVert #1 \right\rVert}

\newcommand{\vph}{\varphi}

\newcommand{\R}{\mathbb{R}}

\begin{document}

\begin{frontmatter}

\title{A Correction Function-based KFBI Method for Brinkman Interface Problems}

\author[upenn]{Han Zhou}
\ead{hzhou24@sas.upenn.edu}

\author[sjtu]{Wenjun Ying\corref{cor1}}
\ead{wying@sjtu.edu.cn}

\cortext[cor1]{Corresponding author}

\address[upenn]{Department of Mathematics, University of Pennsylvania, Philadelphia, PA 19104, USA}
\address[sjtu]{School of Mathematical Sciences, MOE-LSC and Institute of Natural Sciences, Shanghai Jiao Tong University, Minhang, Shanghai 200240, China}

\begin{abstract}
In this work, we propose a correction-function-based kernel-free boundary integral (CF-KFBI) method for solving Stokes- and Brinkman-type interface problems. We begin by recasting the original interface problem with discontinuous coefficients as boundary integral equations, in which the integral operators can be interpreted as boundary data for potential functions that satisfy simpler interface problems without coefficient discontinuities. Each such interface problem is discretized using a corrected Marker-and-Cell (MAC) scheme. Within a narrow band around the interface, we introduce a local correction function that represents the solution jump, leading to a local Cauchy problem. This problem is solved with a collocation method, for which we provide criteria for a minimal choice of collocation points and prove solvability. Several numerical experiments, including both fixed- and moving-interface problems, are presented to demonstrate the accuracy and efficiency of the proposed method.
\end{abstract}


\end{frontmatter}

\section{Introduction}
Many applications in science and engineering require the solution of Brinkman-type interface problems in domains containing materials with sharply varying viscosity, porosity, and permeability~\cite{durlofsky1987analysis,auriault2009domain,Brinkman1949, NieldBejan2017,Pozrikidis1992, BarthesBiesel2016}. 
Such models arise, for example, in free-flow/porous-media coupling for groundwater transport and enhanced oil recovery, where the Brinkman equations interpolate between Darcy flow for the porous media and the incompressible Stokes equations for the viscous fluid~\cite{NieldBejan2017,Lesinigo2011DarcyBrinkman}. 
Related formulations appear in filtration, tissue perfusion, and other biological transport processes involving discontinuous viscosity and interfacial stresses~\cite{Kahshan2020DarcyBrinkmanFiltration, Rohan2021BiotDarcyBrinkman, KhaledVafai2003BioTissues}. 
Across these applications, the presence of discontinuous coefficients, geometric complexity, and singular interface forces introduces substantial numerical challenges and motivates the development of robust solvers.

Traditional numerical methods for interface problems of elliptic partial differential equations (PDEs) typically employ body-fitted meshes that conform to the interface geometry and are combined with finite element or finite difference discretizations~\cite{Babuska1970,Guyomarch2009,Bramble1996,Wang2004,xie2008uniformly,LiMelenkWohlmuthZou2010FEMInterface}. In the setting of two-phase incompressible flows, interface-fitted methods have been developed and studied in depth; see, for instance,~\cite{GrossReusken2011TwoPhase} for a systematic treatment of interface-fitted FEM formulations, curvature forces, and surface-tension effects. Although these body-fitted methods can deliver highly accurate geometric resolution, they rely on high-quality mesh generation and often require remeshing as the interface deforms, which significantly complicates their implementation and efficiency in higher dimensions and in problems involving large deformations or even topological changes.

To avoid remeshing, many unfitted or Cartesian-grid methods have been proposed for interface problems. 
On the finite-difference side, the immersed boundary (IB) method represents interfacial forces as regularized delta sources on a fixed Eulerian grid and has been widely used for fluid–structure interaction in incompressible viscous flows~\cite{Peskin2002IB}. 
Sharp-interface variants instead modify the discrete operators near the interface to enforce jump conditions: the immersed interface method (IIM) and its extensions~\cite{LeVequeLi1994IIM,LeVequeLi1997IIMStokes,Li2003IIMOverview,LiIto2006IIMBook}, the matched interface and boundary (MIB) method~\cite{ZhouZhaoFeigWei2006MIB,YuZhouWei2007MIB,XiaWei2014MIBGalerkin}, and the correction function method (CFM)~\cite{MarquesNaveRosales2011CFM}. 
Related ghost-fluid, embedded-boundary/cut-cell, and difference-potential formulations embed discontinuities and interface conditions into level-set or implicitly described geometries on simple background meshes~\cite{FedkiwAslamMerrimanOsher1999GFM,JohansenColella1998EmbeddedBoundary,AlbrightEpshteynMedvinskyXia2017DPM}, but they typically rely on intricate local constructions that are often tailored to specific PDEs or jump conditions. 
On the finite-element side, extended and generalized FEM (XFEM/GFEM) enrich standard polynomial spaces by interface-adapted basis functions to capture discontinuities or singularities without mesh conformity~\cite{MoesDolbowBelytschko1999XFEM,BelytschkoGracieVentura2009XFEMReview}, while immersed finite element (IFE) methods construct modified basis functions on background meshes so that the jump conditions are satisfied weakly or strongly at the interface~\cite{LiLinWu2003IFE,LiLinRogers2004IFESpace,LiYang2005IFEElasticity,HeLinLinZhang2013IFEParabolic}. 
These unfitted strategies largely eliminate the need for remeshing and are well suited for moving and highly deforming interfaces; however, the development of robust and efficient solvers for Stokes/Brinkman equations remains an active area of research.

Boundary integral equation (BIE) methods provide an attractive alternative by reducing the problem dimension and enforcing interface conditions via layer potentials. 
For homogeneous Stokes flow, classical single- and double-layer formulations and their variants yield second-kind integral equations with favorable conditioning and have been applied extensively to suspensions and droplet dynamics~\cite{PowerMiranda1987CDL,GreengardKropinski2004StokesPeriodic}. 
Pozrikidis’ monograph gives a comprehensive account of such formulations for linearized viscous flow, including vesicle and capsule dynamics~\cite{Pozrikidis1992BIE}. 
Boundary integral formulations have also been adopted for porous-media interface problems~\cite{ahmadi_cortez_fujioka_2017}. 
For coupled Darcy–Stokes systems, BIE methods have been developed in conjunction with domain decomposition techniques~\cite{TLUPOVA2009158,Boubendir2013,TLUPOVA2022110824}. 
Although high-order quadratures and fast solvers have been constructed for these BIEs in geometrically complex settings, enabling large-scale simulations~\cite{HouLowengrubShelley2001BIM,GreengardKropinski2004StokesPeriodic,YingBirosZorin2006BIE3D,AMBROSE2013168}, accurate and efficient boundary integral implementations still require substantial algorithmic complexity, particularly when off-surface evaluation and volume integrals are needed.

The kernel-free boundary integral method was introduced by Ying~\cite{Ying2007} for efficiently solving elliptic PDEs with complex interfaces, and has since been extended to higher-order schemes and nonlinear problems~\cite{xie2020fourth,zhou2023kernel}. 
In~\cite{ZhouYing2024Correction}, we incorporated a correction function into the KFBI framework to solve Poisson interface problems via near-interface corrections and interpolations, which significantly simplifies the treatment of jump terms, especially in higher dimensions and for high-order schemes. 
There, a polynomial collocation method was proposed to solve the local Cauchy problem for the correction function; however, the solvability of the collocation system, particularly under a carefully chosen minimal set of collocation points, was left open. 
In this work, we generalize the correction-function-based KFBI method of~\cite{ZhouYing2024Correction} to Stokes/Brinkman interface problems. 
We formulate boundary integral equations for different coefficient configurations and solve them efficiently on a staggered grid using a corrected MAC scheme. 
In particular, we present the collocation strategy used to obtain the correction function and analyze minimal choices of collocation points, proving that such a choice yields an invertible linear system. 

The remainder of the paper is organized as follows. 
In Section~\ref{sec2}, we introduce the Brinkman equation and its associated interface jump conditions. 
Section~\ref{sec3} is devoted to the boundary integral formulation and the properties of the resulting boundary integral equations. 
In Section~\ref{sec4}, we present the numerical discretization of these boundary integral equations. 
Section~\ref{sec5} describes the numerical method for solving the local Cauchy problem for the correction function. 
Numerical experiments are reported in Section~\ref{sec6}, and Section~\ref{sec7} concludes with a brief summary and a discussion of possible directions for future work.

\section{Governing equation}\label{sec2}

\begin{figure}[htbp]
    \centering
    \includegraphics[width=0.4\linewidth]{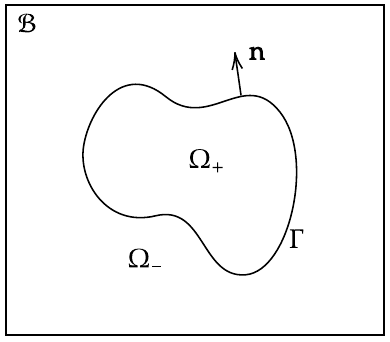}
    \caption{Schematic of the computational domain $\mc B$, the subdomains $\Omega_{\pm}$, and the interface $\Gamma$.}
    \label{fig:domain_setup}
\end{figure}

Let $\mathcal{B}\subset \R^2$ be a rectangular domain and let $\Omega_+ \subset \mathcal{B}$ be a complex-shaped open subdomain embedded in $\mc B$ with a smooth boundary $\Gamma = \partial\Omega_+$, see Figure~\ref{fig:domain_setup}. Assume that $\text{dist}(\Omega_+, \partial\mathcal{B}) \geq d_0 > 0$. Let $\Omega_- = \mathcal{B}\setminus\overline{\Omega_+}$ be the complementary domain. Denote by $\bm n$ the unit normal to $\Gamma$ pointing from $\Omega_+$ to $\Omega_-$.
We are interested in the following interface problem for an incompressible fluid flow
\begin{subequations}\label{eqn:stokes-ip}
\begin{align}
    \mu \Delta \bm u - \kappa \bm u - \grad p &= \bm f, \quad \grad\cdot \bm u = 0, \quad \text{ in } \Omega_+ \cap \Omega_-, \\
    \jump{ \bm u } &= \bm g_1,\quad  \jump{\bm\sigma_\mu \bm n} = \bm g_2 , \quad\text{ on }\Gamma,
\end{align}
\end{subequations}
where the stress tensor $\bm \sigma_\mu$ is given by $\bm \sigma_\mu(\bm u, p) = 2\bm D(\bm u)- p  I $ with $\bm D(\bm u) = \frac{1}{2}(\grad \bm u + \grad \bm u^T) $ being the rate-of-deformation tensor. We will denote by $\bm \sigma$ the stress tensor $\bm\sigma_\mu$ with $\mu = 1$.
The coefficient $\mu > 0$ is the viscosity of the fluid.
For $\kappa=0$, this is the two-phase Stokes equation, and for $\kappa > 0$, the problem is the two-phase Brinkman equation describing fluid flow through porous media. It can also arise from the time discretization of the two-phase unsteady Stokes or Navier-Stokes equations.
In this work, we consider the case where the coefficients are piecewise constant, i.e., $\mu=\mu(\bm x)=\mu_\pm$ and $\kappa=\kappa(\bm x)=\kappa_\pm$ for $\bm x\in \Omega_\pm$.
For a piecewise continuous function $f$, the jump $\jump{\cdot}$ is defined as
\begin{equation}
    \jump{f}(\bm x) = \lim_{\varepsilon\to 0+} (f(\bm x - \varepsilon\bm n) - f(\bm x + \varepsilon \bm n)), \quad \bm x \in \Gamma.
\end{equation}
The function $\bm f\in L^2(\mc B)$ represents a given body force.
The functions $\bm g_1 \in H^{\frac{1}{2}}(\Gamma)$ and $\bm g_2\in H^{-\frac{1}{2}}(\Gamma)$ represent the jumps in velocity and normal stress on $\Gamma$, respectively. For an impermeable interface, it is common to set $\bm g_1 = \bm 0$, meaning that the velocity field is continuous across the interface. The function $\bm g_2$ arises from surface tension or the bending force exerted on the fluid by the interface. The boundary condition on $\partial\mathcal{B}$ is chosen to be the no-slip condition. Other boundary conditions, such as inhomogeneous velocity boundary conditions or bi-periodic ones, can also be considered. Since $\partial\mathcal{B}$ has a regular shape, boundary conditions on it are relatively easy to handle compared with those on the complex interface.

The velocity $\bm u$ is uniquely determined by an integration-by-parts argument together with the Dirichlet boundary condition on $\partial\mc B$. However, the pressure $p$ is only determined up to an additive constant. Therefore, we impose the additional constraint
$\int_{\mc B} p\,d\bm x = 0$
to seek the solution $p$ with mean zero. 
In this way, the solution to~\eqref{eqn:stokes-ip} is uniquely determined.

\section{Boundary integral formulations}\label{sec3}
\subsection{Integral operators}
Let $\delta$ be the Dirac delta function. Given a constant $c^2 \geq 0$, for every fixed $\bm x \in\mathcal{B}$, let $(G_{\bm u }, G_p)$ be the Green function pairs defined by
\begin{equation}
\begin{aligned}
    (\Delta_{\bm y} - c^2) G_{\bm u }(\bm y , \bm x ) - \nabla_{\bm y} G_p(\bm y , \bm x ) &= \delta(\bm y - \bm x )  I , \quad \bm y \in\mathcal{B},\\
    \nabla_{\bm y} \cdot G_{\bm u }(\bm y , \bm x ) &= 0,  \quad \bm y \in \mathcal{B}, \\
    G_{\bm u}(\bm y, \bm x) & = \bm 0,  \quad \bm y \in \partial\mathcal{B}.
\end{aligned}
\end{equation}
with the mean zero constraint $\int_{\mc B} G_p (\mb y,\bm x)\,d\bm y = 0$.
Due to the boundary condition on $\mc B$, the analytic expressions of $G_{\bm u}, G_p$ may not be available. They are closely related to the free-space Green's functions $(G_{\bm u}^f, G_p^f)$; particularly, they only differ by a smooth function since $(G_{\bm u}-G_{\bm u}^f, G_p - G_p^f)$ satisfies the homogeneous Brinkman equation in $\mc B$, whose solution is smooth.

Let $\boldsymbol \vph $ and $\boldsymbol \psi $ be two surface density functions defined on $\Gamma$.
For $\bm y\in \mathcal{B}$, define single-layer potentials 
\begin{align}
      V_{\bm u } \bm \psi (\bm y) &= \int_{\Gamma} G_{\bm u }(\bm y , \bm x ) \bm \psi (\bm x ) d s_{\bm x }, \\ 
      V_p \bm \psi (\bm y) &= \int_{\Gamma} G_p(\bm y , \bm x ) \bm \psi (\bm x ) d s_{\bm x },
\end{align}
and double-layer potentials
\begin{align}
      D_{\bm u} \bm \vph (\bm y) &= \int_{\Gamma} \bm \vph (\bm x )\cdot \bm{\sigma}\paren{ G_{\bm u }(\bm y , \bm x ), G_p(\bm y , \bm x ) }\bm{n}(\bm x ) d s_{\bm x }, \\ 
      D_p \bm \vph (\bm y)& = 2 \int_{\Gamma} \frac{\partial G_p(\bm y , \bm x )}{\partial \bm{n}_{\bm x }} \bm \vph (\bm x ) d s_{\bm x } .
\end{align}
Let $\bm{f}$ be a function defined in $\mathcal{B}$. We also define volume potentials
\begin{align}
    N_{\bm u } \bm{f}(\bm{p}) &= \int_{\mathcal{B}}  G_{\bm u }(\bm{p}, \bm{q})\bm{f}(\bm{q}) d \bm{q}, \\
    N_p \bm{f}(\bm{p}) &= \int_{\mathcal{B}}  G_p(\bm{p}, \bm{q}) \bm{f}(\bm{q}) d \bm{q} .
\end{align}

Note that the Green functions $(G_{\bm u}, G_p)$ defined in $\mathcal{B}$ with a boundary condition usually do not admit an analytic expression. They differ from those defined in free space, also referred to as the Stokeslet for $\kappa=0$, only by a smooth function. 
Let $\gamma^\pm$ be the trace of piecewise continuous functions on $\Gamma$ from the side of $\Omega^\pm$, respectively.
With the jump relation of the potential functions known, we can define the single-layer, double-layer, adjoint double-layer, and hyper-singular boundary integral operators $\mc V, \mc D, \mc D', \mc H$ for $\bm x \in\Gamma$. Then, we define the following integral operators
\begin{align}
      \mc V \bm \psi (\bm x) &= \frac{1}{2}\paren{\gamma^+  V_{\bm u }\bm \psi (\bm x ) + \gamma^-  V_{\bm u }\bm \psi (\bm x )},  \label{eqn:sl-bi} \\
      \mc D'\bm \psi (\bm x) &= \frac{1}{2}\paren{\gamma^+  \bm{\sigma}(  V_{\bm u }\bm \psi ,  V_{p}\bm \psi )(\bm x ) \bm{n}(\bm x) + \gamma^-  \bm{\sigma}(  V_{\bm u }\bm \psi ,  V_{p}\bm \psi )(\bm x ) \bm{n}(\bm x)}, \label{eqn:adl-bi}  \\
     \mc D  \bm \vph (\bm x) &= \frac{1}{2}\paren{ \gamma^+  D_{\bm u }\bm \vph (\bm x ) + \gamma^-  D_{\bm u }\bm \vph (\bm x ) }, \label{eqn:dl-bi} \\
     \mc H  \bm \vph (\bm x) &= \frac{1}{2}\paren{ \gamma^+ \bm{\sigma} ( D_{\bm u }\bm \vph ,  D_{p}\bm \vph )(\bm x )\bm{n}(\bm x) + \gamma^- \bm{\sigma} ( D_{\bm u }\bm \vph ,  D_{p}\bm \vph )(\bm x )\bm{n}(\bm x)}, \label{eqn:hs-bi}  \\
     \mc N_D \bm f(\bm x) & = \frac{1}{2}\paren{\gamma^+ N_{\bm u}\bm f(\bm x) + \gamma^- N_{\bm u}\bm f(\bm x) }, \label{eqn:vp-D}\\
     \mc N_N \bm f(\bm x) & = \frac{1}{2}\paren{\gamma^+ \bm \sigma(N_{\bm u}\bm f, N_p \bm f)(\bm x) + \gamma^- \bm \sigma(N_{\bm u}\bm f, N_p \bm f)(\bm x)} .\label{eqn:vp-N}
\end{align}
Now, let us reduce the interface problem~\eqref{eqn:stokes-ip} to a boundary integral equation.
We redefine $\wt p = p/\mu$,  $\wt{\bm \sigma} = \grad \bm u + \grad \bm u^T - \wt p I$ and $\wt {\bm f} = \bm f/\mu$. 

Therefore, the Green function defined here has the same singularity behavior as the free-space one and, thus, the potential functions defined using it have the same jump relations on the interface $\Gamma$. Let $(\bm v, q)$ be a potential function pair; by classical potential theory, we know that it satisfies the interface problem with constant coefficients
\begin{subequations}\label{eqn:equiv-ip}
\begin{align}
    \Delta  \bm v - c^2 \bm v - \grad q &= \bm F, \quad \grad \cdot \bm v = 0,\quad \text{ in }\mathcal{B}\setminus\Gamma, \\
    \jump{\bm v} &= \bm \Phi, \quad  \jump{\bm \sigma \bm n} = \bm \Psi, \quad \text{ on }\Gamma,\\
    \bm v &= \bm 0, \quad \text{ on }\partial\mathcal{B},
\end{align}    
\end{subequations}
with $\int_{\mc B}q\,d\bm x = 0$.
Here, the data functions $\bm F, \bm \Phi, \bm \Psi$ are specified as follows:
\begin{itemize}
    \item For $\bm v = V_{\bm u}\bm\psi, q = V_p \bm\psi$, we set $\bm F = \bm 0$, $\bm \Phi = \bm 0$, $\bm \Psi =-\bm\psi$,
    \item For $\bm v = D_{\bm u}\bm\psi, q = D_p \bm\psi$, we set $\bm F = \bm 0$, $\bm \Phi = \bm \vph$, $\bm \Psi =\bm 0$,
    \item For $\bm v = N_{\bm u}\bm\psi, q = N_p \bm\psi$, we set $\bm F = \bm f$, $\bm \Phi = \bm 0$, $\bm \Psi =\bm 0$.
\end{itemize}

\subsection{Boundary integral equations}
First, we consider a simpler case $c^2 = \kappa_+/\mu_+ = \kappa_-/\mu_-$.
Let $\bm \psi = \jump{\bm \sigma(\bm u, p/\mu)\bm n}$ be an unknown density function and $\wt{\bm f} = \bm f/\mu$. Then, the solution can be represented as
\begin{equation}
    \bm u = D_{\bm u} \bm g_1 - V_{\bm u} \bm \psi  + N_{\bm u} \wt {\bm f}, \quad \wt p = D_p \bm g_1 - V_p \bm \psi + N_p \wt {\bm f},
\end{equation}
where the potentials are defined using the Green function associated with the constant $c^2 = \kappa/\mu$.
Using the jump condition $\jump{\mu \bm \sigma(\bm u, p/\mu)\bm n} = \bm g_2$, we can obtain a Cauchy singular integral equation of the second kind for $\bm \psi$,
\begin{equation}\label{eqn:bie-1}
    \bm \psi + 2 A_\mu \mc D'\bm \psi = \frac{2}{\mu_++\mu_-}\bm g_2 + 2 A_\mu (\mc H \bm g_1 + \mc N_N \wt {\bm f}) ,
\end{equation}
where $A_\mu = (\mu_- - \mu_+)/(\mu_- + \mu_+) \in (-1,1)$.
The adjoint-double layer integral is a Cauchy singular integral; although not compact in general, the Fredholm alternative is still valid~\cite{HsiaoWendland2008}. Therefore, the uniqueness of the solution implies the solvability of the equation.
\begin{proposition}\label{thm:uniq-1}
    The boundary integral equation~\eqref{eqn:bie-1} has a unique solution.
\end{proposition}
\begin{proof}
    We only need to prove that the homogeneous equation of~\eqref{eqn:bie-1} has only the zero solution. Recall the jump relations of the single-layer potential $(\bm v, q)=(-V_{\bm u}\bm\psi, -V_p\bm \psi)$; then, the homogeneous equation of~\eqref{eqn:bie-1} implies that the single-layer potential satisfies the homogeneous interface problem~\eqref{eqn:stokes-ip} with $\jump{\bm v} =\bm 0$ and $\jump{\mu(\grad \bm v+\grad\bm v^T) - qI}\bm n = \bm 0$. Applying integration by parts and recalling the boundary condition $\bm v = \bm 0$ on $\partial\mc B$, we have
    \begin{equation}
        \mu_+\int_{\Omega_+}|D(\bm v)|^2\,\bm x + \mu_-\int_{\Omega_-}|D(\bm v)|^2\,\bm x = 0.
    \end{equation}
    Hence, $\bm v \equiv\bm 0$ and $q\equiv C$ in $\mc B$. In fact, due to the mean zero constraint, $C = 0$. As a result $\bm \psi = \jump{(\grad \bm v+\grad\bm v^T) - qI}\bm n = \bm 0$.
\end{proof}

For the generic coefficient case $\kappa_+/\mu_+ \neq \kappa_-/\mu_-$, let $\bm \vph$ and $\bm \psi$ be two unknown density functions. We exploit the idea of domain decomposition to represent the interior/exterior solutions separately with different potential functions. The interior solution is expressed as a sum of the double-layer and volume potentials
\begin{equation}
    \bm u\chi_{\Omega_+} = D_{\bm u}^+ \bm \vph + N^+_{\bm u} \wt{\bm f}\chi_{\Omega_+},\quad \wt p \chi_{\Omega_+} = D_p^+ \bm \vph + N^+_p \wt{\bm f}\chi_{\Omega_+},
\end{equation}
and the exterior solution as a sum of the single-layer and volume potentials
\begin{equation}
    \bm u\chi_{\Omega_-} = -V_{\bm u}^- \bm \psi + N^-_{\bm u} \wt{\bm f}\chi_{\Omega_-},\quad \wt p \chi_{\Omega_-} = -V_p^- \bm \psi + N^-_p \wt{\bm f}\chi_{\Omega_-},
\end{equation}
where the superscripts $\pm$ indicate that the potentials are defined using the Green function associated with the constant $c_{\pm}^2 = \kappa_\pm / \mu_\pm$.
With the above ansatz on the solution representation and applying the jump conditions of the interface problem~\eqref{eqn:stokes-ip}, we can obtain the following boundary integral system
\begin{subequations}\label{eqn:bie-2}
\begin{align}
    \frac{1}{2} \bm \vph  + \mc{D}^+ \bm \vph  + \mc{V}^-\bm \psi  &= \bm g_1 + \mc{N}_D^-(\wt{\bm{f}}\chi_{\Omega_-}) - \mc{N}_D^+(\wt{\bm{f}}\chi_{\Omega_+}),  \\
    \frac{1}{2}\bm \psi  + \frac{\mu^+}{\mu^-}\mc{H}^+ \bm \vph  + \mc{D}^{\prime,-} \bm \psi  &= \frac{1}{\mu^-}\bm g_2 + \mc N_N^- (\wt{\bm{f}}\chi_{\Omega_-}) - \frac{\mu^+}{\mu^-}\mc N_N^+ (\wt{\bm{f}}\chi_{\Omega_+}).    
\end{align}    
\end{subequations}
Since the integral operators $\mc V^-, \mc D^{\prime,-}, \mc D^+, \mc H^+$ are Fredholm operators of index zero~\cite[Theorem 7.6,7.8]{mclean2000strongly}, we can also apply the Fredholm alternative to demonstrate the solvability of the boundary integral equation~\eqref{eqn:bie-2}. 
\begin{proposition}
    The boundary integral equation~\eqref{eqn:bie-2} has a unique solution.
\end{proposition}
\begin{proof}
    
Let $\bm v = (D_{\bm u}^+ \bm \vph)\chi_{\Omega_+} - (V_{\bm u}^- \bm \psi)\chi_{\Omega_-}$ and $q = (D_p^+ \bm \vph)\chi_{\Omega_+} - (V_p^- \bm \psi)\chi_{\Omega_-}$. The homogeneous case of~\eqref{eqn:bie-2} implies that $(\bm v, q)$ satisfies the homogeneous interface problem~\eqref{eqn:stokes-ip} with $\jump{\bm v} = \jump{\mu(\grad \bm v+\grad\bm v^T) - qI}\bm n = \bm 0$. Similarly, applying integration by parts and the mean zero constraint, we have $\bm v \equiv\bm 0$ and $q\equiv C = 0$ in $\mc B$.

Recall the jump relation of the double-layer potential; we have $\gamma^-\bm\sigma(D_{\bm u}^+ \bm \vph,D_p^+\bm \vph)\bm n = \gamma^+ \bm\sigma(D_{\bm u}^+ \bm \vph,D_p^+\bm \vph)\bm n = \bm 0$ on $\Gamma$. Then $(\bm v', q')=((D_{\bm u}^+\bm\psi)\chi_{\Omega^-}, (D_p^+\bm \psi)\chi_{\Omega^-})$ satisfies a homogeneous Brinkman boundary value problem in $\Omega^-$ with boundary conditions $\bm\sigma(\bm v', q')\bm n = C\bm n$ on $\Gamma$ and $\bm v' = 0$ on $\partial\mc B$. 
Applying integration by parts shows that $\bm v' \equiv0$ and $q' \equiv 0$. 

Using the jump relation of the single-layer potential, $\gamma^+ V_p^-\bm \psi = \gamma^- V_p^-\bm \psi \equiv \bm 0$. Then $(\bm v'', q'')=((-V_{\bm u}^-\bm\psi)\chi_{\Omega^+}, (-V_p^-\bm \psi)\chi_{\Omega^+})$ satisfies the homogeneous Brinkman boundary value problem in $\Omega_+$ with the boundary condition $\bm v'' = \bm 0$ on $\Gamma$. As a result, $\bm v'' \equiv \bm 0$ and $q'' \equiv C'$. Additionally, since 
    \begin{equation}
        0 = \int_{\mc B} -V_p^-\bm \psi\,d\bm x = - \int_{\Omega^+} V_p^-\bm \psi\,d\bm x-\int_{\Omega^-} V_p^-\bm \psi\,d\bm x = C ' \int_{\Omega^+} 1 \,d\bm x.
    \end{equation}
    Hence $C' = 0$.
    Therefore, $(D_{\bm u}^+\vph , D_p^+\bm\vph) = (V_{\bm u}^-\bm\psi , V_p^-\bm \psi) \equiv (\bm 0, 0)$ and $\bm \vph = \jump{D_{\bm u}^+\vph} = \bm 0$ and $\bm \psi = -\jump{\bm\sigma(V_{\bm u}^-\bm\psi, V_p^-\bm\psi)} = \bm 0$.
\end{proof}

Conventional boundary integral methods treat the integral operators as boundary or volume integrals with singular kernels and rely on numerical quadrature for their approximation. However, the potential functions $V, D, N$ provide a different yet equivalent interpretation of these operators through the interface problem~\eqref{eqn:equiv-ip}. It may seem that we have transformed a PDE interface problem into boundary integral equations only to reformulate it again as a PDE interface problem. However, the interface problem~\eqref{eqn:equiv-ip} is actually simpler than the original one, since it has constant coefficients and can therefore be solved efficiently and accurately with fast PDE solvers. In this way, we completely avoid the use of analytic Green's function expressions as integral kernels, making the method kernel-free.

\section{Numerical method}\label{sec4}

\subsection{Discretization of the boundary integral equation}
We mainly discuss the discretization of the boundary integral equation~\eqref{eqn:bie-1}, since the treatment of~\eqref{eqn:bie-2} is similar.
Suppose that the interface $\Gamma$ is a smooth closed curve, parametrized by $\bm X(s): \mathbb R/(2\pi\mathbb Z) \to \mathbb R^2$. Assume that the parameter $s$ is proportional to arclength. Let $M$ be a positive integer and let $\Delta s = 2\pi/M$. By uniformly partitioning the curve in the parameter $s$, we obtain a set of quasi-uniform interface points $\bm X_i = \bm X(s_i)$ for $i = 0,1,\cdots, M-1$. 
Let $L_i(s)$ be the interpolation basis function, such as periodic cubic splines or the Fourier basis.
Let $\bm \psi^M(s)$ be the approximation to the pullback $(\bm\psi \circ \bm X) (s)$ given in the form
\begin{equation}
    \bm \psi^M(s) = \sum_{i = 0}^{M-1} \bm\psi_i L_i(s),\quad L_i(s_j) =
    \begin{cases}
        1 & j = i\\
        0 & j \neq i.
    \end{cases},\quad
    \sum_{i = 0}^{M-1} L_i(s) \equiv 1.
\end{equation}
where $\bm \psi_i$ is the unknown nodal value of $\psi^M$ at the point $\bm X_i$. 
Denote by $\wt{\bm \psi}^M(\bm X(s)) = \bm \psi^M(s)$.
Let $\mc D'_h, \mc H_h, \mc N_{M,h}$ be the numerical approximation of $\mc D', \mc H, \mc N_N$ on a background Cartesian grid with discretization parameter $h$. 
The computation of $\mc D'_h, \mc H_h, \mc N_{N,h}$ is performed by solving the interface problem~\eqref{eqn:equiv-ip}. Once the interface problem \eqref{eqn:equiv-ip} is solved on a background Cartesian grid, one can apply a specialized interpolation scheme to obtain the integral values \eqref{eqn:sl-bi}-\eqref{eqn:vp-N} on the interface.

Then, let the boundary integral equation be satisfied at $\bm X_i$ for $i=0,1,\cdots,M-1$; the discrete scheme for \eqref{eqn:bie-1} can be written as
\begin{equation}
    \bm \psi^M (s_i) + 2 A_\mu (\mc D'_h\wt{\bm \psi}^M)(\bm X_i) = \frac{2}{\mu_+ +\mu_-} \bm g_2 (\bm X_i) + 2 A_\mu ((\mc H_h \bm g_1)(\bm X_i) + (\mc N_{N,h} \wt {\bm f})(\bm X_i)),
\end{equation}
and can be written in matrix form
\begin{equation}
    (\bm I + 2 A_\mu \bm K) \bm \psi = \bm b,
\end{equation}
where $\bm \psi$ now represents the vector consisting of the pointwise values $\bm \psi_i$ for $i=0,1,\cdots,M-1$.
The Krylov subspace methods are robust and efficient for the above linear systems. We will apply the Generalized Minimum Residual (GMRES) method to iteratively solve the system.
For linear systems arising from well-conditioned boundary integral equations, the discrete algebraic system mimics the so-called numerical well-conditionedness; the GMRES method converges efficiently even without sophisticated preconditioners. 
We also note that the iterative method can be implemented in a matrix-free manner. One only needs to implement the action of $\bm K$ on a given vector instead of assembling the full matrix, which is essentially composed of boundary/volume integral evaluations.

Solving the constant interface problem~\eqref{eqn:equiv-ip} is a main ingredient of the kernel-free boundary integral method. The main difficulty in solving such an interface problem is that the solution may have low regularity at the interface, which complicates the numerical approximation. To this end, we will introduce an efficient PDE solver for the problem, based on the correction function.

\subsection{Correction function method}
\subsubsection{Correction function revisit}
We first revisit the idea of the correction function for the simple 1D case, where $\Gamma = \{\gamma\}$ with $\gamma \in (-1,1)$ and $\Omega_\pm$ are two intervals, assumed to be $(\pm 1, \gamma)$.
Let $u$ be a piecewise continuous function defined on $\Omega_+\cup\Omega_-$.
Assume that $\at{u}{\Omega_\pm}\in C^{k}(\overline{\Omega_\pm})$. Since $u$ is not smooth at the point $\Gamma$, any standard finite difference taken across the interface will lead to a large error. 
For example, let $k=3$, $x\in\Omega_-$, and $x+h\in\Omega_+$, then
\begin{equation}
\begin{aligned}
    u(x) &= \sum_{m = 0}^k \frac{u^{(m)}(\gamma-)}{m!} (x - \gamma)^m + \mc{O}((x-\gamma)^{k}), \\ 
    u(x+h) &= \sum_{m = 0}^k \frac{u^{(m)}(\gamma+)}{m!} (x +h - \gamma)^m + \mc{O}((x+h-\gamma)^{k}).
\end{aligned}
\end{equation}
Assume that $x+\frac{h}{2} > \gamma$; by Taylor expansion at the interface point $\gamma$, the centered difference approximation of $u'\paren{x+\frac{h}{2}}$ yields
\begin{equation}\label{eqn:fd-1}
\begin{aligned}
    &\,\,\frac{u(x+h) - u(x)}{h} - u'\paren{x+\frac{h}{2}} \\
    &= \frac{1}{h} \sum_{m = 0}^3 \paren{ \frac{u^{(m)}(\gamma+)}{m!} (x + h - \gamma)^m -  
    \frac{u^{(m)}(\gamma-)}{m!} (x - \gamma)^m } \\
    &\quad - \sum_{m = 1}^3  \frac{u^{(m)}(\gamma+)}{(m-1)!} (x + \frac{h}{2} - \gamma)^m  + \mathcal{O}(h^{3}) \\
    & = \frac{1}{h}\paren{\jump{u}(\gamma) + (x-\gamma) \jump{u'}(\gamma) + \frac{(x-\gamma)^2}{2}\jump{u^{\prime\prime}}(\gamma)} + \mathcal{O}(h^2).
\end{aligned}
\end{equation}
We can see that the error is of order $\mathcal{O}(h^{-1})$ and does not converge.
A simple remedy is to incorporate the large error into the finite difference approximation as correction terms, such that the resulting error becomes $\mathcal{O}(h^2)$, which is the expected convergence rate. This is the idea of defect correction. It can be applied to PDE interface problems where the solution is only expected to be piecewise smooth, leading to Mayo's method and immersed interface methods.

Notice that although the expansion formula for $f(x), f(x+h), f\paren{x+\frac{h}{2}}$ at $\gamma$ is complicated, the resulting correction is actually simple and can be derived in a different way. Now, suppose we can extend $\at{u}{\Omega_+}$ to $[-1,1]$ such that the extended function, denoted by $u_+$, belongs to $C^{k}([-1,1])$. Also, we can extend $\at{u}{\Omega_-}$ to $u_-\in C^{k}([-1,1])$. Then $u$ can be represented as
\begin{equation}
    u(x) = u_+(x) \chi_{\Omega_+}(x) + u_-(x) \chi_{\Omega_-}(x), \quad x\in \Omega_+\cup\Omega_-,
\end{equation}
where $\chi_{\Omega_\pm}$ is the characteristic function on the set $\Omega_\pm$.
Then, the finite difference leads to
\begin{equation}\label{eqn:fd-2}
\begin{aligned}
    \frac{u(x+h) - u(x)}{h} - u'\paren{x+\frac{h}{2}} &= \frac{u_+(x)- u_-(x)}{h} + \frac{u_+(x+h) - u_+(x)}{h} \\
     - u_+'\paren{x+\frac{h}{2}} &= \frac{u_+(x)- u_-(x)}{h} + \mathcal{O}(h^2).
\end{aligned}
\end{equation}
Comparing this with \eqref{eqn:fd-1} and noting that
\begin{equation}\label{eqn:jc}
    \jump{u^{(m)}}(\gamma) = u_+^{(m)}(\gamma) - u_-^{(m)}(\gamma),\quad m=0,1,2,
\end{equation}
we can see the equivalence between these two since \eqref{eqn:fd-1} can be obtained by simply expanding the function $u_+ - u_-$ at the point $\gamma$.
Denote by $\wh u = u_+ - u_-$ the correction function. Then, as long as one can determine the correction function, one can easily obtain a modified finite difference approximation for piecewise continuous functions. Also, since the finite difference is usually taken for small $h$, and the modification is only needed at points near the interface, seeking a $\wh u$ locally in the vicinity of $\Gamma$ with a distance $\mathcal{O}(h)$ is sufficient. This locality reduces the difficulty of extending a function. Inspired by \eqref{eqn:fd-1}, the simplest choice is to represent $\wh u$ as a polynomial
\begin{equation}
    \wh u(x) = \jump{u}(\gamma) + (x-\gamma) \jump{u'}(\gamma) + \frac{(x-\gamma)^2}{2}\jump{u^{\prime\prime}}(\gamma),
\end{equation}
which also satisfies the jump conditions \eqref{eqn:jc}.

\subsubsection{Local Cauchy problem}
In order to solve the interface problem~\eqref{eqn:equiv-ip} with a finite difference scheme, we first seek a correction function to represent the correction needed. 
For convenience, we assume that the interface $\Gamma$ and the data functions $\at{\bm F}{\Omega_\pm}, \bm \Phi, \bm \Psi$ belong to $C^{\infty}$ so that the solution $(\at{\bm v}{\Omega_\pm}, \at{q}{\Omega_\pm})$ can be extended to smooth functions $(\bm v^\pm, q^\pm)$ in a narrow band $\Omega_\Gamma\subset\mathcal{B}$ close to $\Gamma$. Denote the correction function by $(\wh{\bm v}, \wh q)$.
Since finite difference schemes usually have local stencils whose diameter is of order $\mc O(h)$, where $h$ is the grid spacing in $\mathcal{B}$, the width of the narrow band only needs to satisfy
\begin{equation}
    \sup_{\bm x\in\Omega_\Gamma}\text{dist}(\bm x, \Gamma) \leq C h,
\end{equation}
where the constant $C$ depends only on the finite difference scheme and the shape of $\Gamma$.
Assume that the piecewise function $\at{\bm F}{\Omega_\pm}$ can also be smoothly extended to smooth functions $\bm F^\pm\in C^\infty(\Omega_\Gamma)$. Denote by $\wh {\bm F} = \bm F^+ - \bm F^-$. 
If the extension of the solution is defined with the Cauchy problem of the Brinkman equation given the Cauchy data on $\Gamma$ and right-hand sides $\bm F^\pm$, then the correction function $(\wh{\bm v}, \wh q)$ satisfies the two-sided Cauchy problem as follows
\begin{subequations}\label{eqn:cau-pro}
\begin{align}
    \Delta  \wh{\bm v} - c^2 \wh{\bm v} - \grad \wh q &= \wh{\bm F}, \quad \grad \cdot \wh{\bm v} = 0,\quad \text{ in }\Omega_\Gamma, \\
    \wh{\bm v} &= \bm \Phi, \quad \wh{\bm \sigma} \bm n = \bm \Psi, \quad \text{ on }\Gamma.
\end{align}    
\end{subequations}
It is worth mentioning that the above Cauchy problem is not well posed in the sense of Hadamard. For a flat interface, by applying the Fourier transform, we see that the growth of the solution is exponential, $\mathcal{O}(\varepsilon e^{|k|d})$, where $\varepsilon$ and $k$ are the amplitude and frequency of the initial data on $\Gamma$, respectively, and $d$ is the distance from $\Gamma$. This implies that small perturbations in the data, namely the initial conditions $\bm \Phi$ and $\bm \Psi$, can lead to unbounded growth away from $\Gamma$. However, for the purpose of numerically solving an interface problem with a local stencil, we solve the problem only in $\Omega_\Gamma$, so that $d = \mc O (h)$. In addition, if the initial data has compactly supported Fourier coefficients, then the error can be bounded in $\Omega_\Gamma$. This is usually the case in our method, since the data functions we use are polynomial interpolants based on finitely many points on $\Gamma$, which admit a frequency-cutoff property. For example, suppose the mesh size on $\Gamma$ is $\Delta s$; Fourier interpolation using values at mesh nodes has a maximum frequency of order $\mc O((\Delta s)^{-1})$. As long as one sets the background and interface grid spacings so that $h \le C \Delta s$ with a constant $C$, then it holds that
\begin{equation}
    e^{|k|d} \leq  e^{\frac{h}{\Delta s}} \leq e^C.
\end{equation}
Hence, the error is bounded.
We will introduce the numerical method for solving the local Cauchy problem in the next section.

\subsection{Corrected MAC scheme}
In the absence of interfaces, problem \eqref{eqn:equiv-ip} can be solved with the Marker-and-Cell (MAC) finite difference scheme on a staggered grid.
However, the standard MAC scheme leads to large local truncation errors near the interface due to the non-smoothness of the solution at the interface.
We correct the MAC scheme by incorporating the local truncation errors utilizing the correction function.
For simplicity, we assume that the rectangular domain $\mathcal{B}$ is a unit square $(0,1)^2$.

First, the domain $\mathcal{B}$ is uniformly partitioned into a Cartesian grid in each spatial direction. The grid nodes are defined as $(x_i, y_j)$ for $i=0,1,\cdots, N_x$, $j=0,1,\cdots, N_y$ with $N_x = N_y$ and $x_{i+1}-x_i = y_{j+1}-y_j = h$. 
Define the centered nodes $x_{i-\frac{1}{2}} = x_i - \frac{1}{2}h$, $y_{j-\frac{1}{2}} = y_j - \frac{1}{2}h$ for $i, j = 1, 2, \cdots, N$.
The discrete nodes for the unknown variables $\bm v = (v^{(1)}, v^{(2)})$ and $q$ are defined at different positions. 
We define the grid node sets $\mathcal{T}^1_h, \mathcal{T}^2_h,\mathcal{T}^3_h$ for the discrete approximation of the variables $v^{(1)}, v^{(2)}, p$, respectively,
\begin{align}
    &\mathcal{T}^1_h = \{(x_i, y_{j-\frac{1}{2}}), \quad i=1,\cdots, N-1, \quad j = 1,2,\cdots,N \}, \\
    &\mathcal{T}^2_h = \{(x_{i-\frac{1}{2}}, y_j), \quad j=1,\cdots, N-1, \quad i = 1,2,\cdots,N \},\\
    &\mathcal{T}^3_h = \{(x_{i-\frac{1}{2}}, y_{j-\frac{1}{2}}), \quad i,j = 1,2,\cdots,N \},
\end{align}
We can also define interior and exterior grid nodes by
$
    \mc T_h^{d, \pm} = \mc T_h^d\cap \Omega_\pm .
$

Define the numerical solution $(\bm v_h, q_h)(x_i,y_j) = (\bm v_{i,j}, q_{i,j})$ where $v^{(1)}_{i,j}$ is defined as the numerical approximation of $v^{(1)}$ at the grid node $(x_i, y_{j-\frac{1}{2}})\in\mathcal{T}^1_h$. 
Similarly, we define $v^{(2)}_{i,j}, q_{i,j}$ as numerical approximations to $v^{(2)}, q$ at grid nodes in the grids $\mathcal{T}^2_h$, $\mathcal{T}^3_h$, respectively.
Let $f_{i,j}^{(d)} = f^{(d)}(x_i, y_j)$ for $d=1,2$ and $f_{i, j}^{(3)} = 0$.
For a grid function $v_{h}$, define the first order finite difference operators $\delta_{h,d}^\pm$ for $d=1,2$ by 
\begin{equation}
\begin{aligned}
    \delta_{h,1}^+ v_{i,j} = (v_{i+1,j}-v_{i,j})/h, \quad \delta_{h,1}^- v_{i,j} = (v_{i,j}-v_{i-1,j})/h, \\
    \delta_{h,2}^+ v_{i,j} = (v_{i,j+1}-v_{i,j})/h, \quad \delta_{h,2}^- v_{i,j} = (v_{i,j}-v_{i,j-1})/h,
\end{aligned}
\end{equation}
Define the difference operators $L_h^d$ as follows,
\begin{align}
    L_h^1[\bm v_h, q_h](x_i,y_j) &= \left( \sum_{d=1}^2 \delta_{h,d}^+ \delta_{h,d}^- - c^2 \right) v^{(1)}_{i,j} - \delta_{h,1}^+ q_{i,j}, \\
    L_h^2[\bm v_h, q_h](x_i,y_j) &= \left( \sum_{d=1}^2 \delta_{h,d}^+ \delta_{h,d}^- - c^2 \right) v^{(2)}_{i,j} - \delta_{h,2}^+ q_{i,j}, \\
    L_h^3[\bm v_h, q_h](x_i,y_j) &= \delta_{h,1}^- v^{(1)}_{i,j} + \delta_{h,2}^- v^{(2)}_{i,j}.
\end{align}
The MAC scheme is given by
\begin{equation}
    L_h^d[\bm v_h, q_h](x_i,y_j) = f_{i,j}^{(d)}, \quad \quad  (x_i,y_j)\in \mc T_h^d, d= 1,2,3,
\end{equation}
together with the symmetric treatment for the no-slip boundary condition on $\partial\mathcal{B}$ .
Denote by $\mc S_h(\bm x)\subset \cup_{d=1}^3 \mc T_h^d$ the set of stencil nodes of the finite difference scheme at the grid node $\bm x \in \mc T_h^d$ for $d=1,2,3$.
One can define irregular grid nodes $\mc I_h^d$ as the set of grid nodes at which the finite difference stencil intersects the interface,
\begin{equation}
    \mc I_h^d = \{\bm x \in \mc T_h^d: (\mc S_h(\bm x)\cap \Omega_+)\cup (\mc S_h(\bm x)\cap \Omega_-) \neq \emptyset\}, \quad d=1,2,3.
\end{equation}
Then, the regular grid nodes will be denoted as $\mc T_h^d\setminus\mc I_h^d$.
As regular grid nodes, the finite difference is taken on a smooth function, and the scheme achieves the expected accuracy. However, at irregular grid nodes, corrections are needed.

At a point $\bm x \in \Omega_\pm$, we can write the exact solution as
\begin{equation}
    \bm v(\bm x) = \bm v^\pm(\bm x) \mp  \wh{\bm v}(\bm x) \chi_{\Omega_\mp}(\bm x), \quad  q(\bm x) =  q^\pm(\bm x) \mp  \wh q(\bm x)\chi_{\Omega_\mp}(\bm x).
\end{equation}
Assume that $\bm x\in\Omega_\pm \cap \mc T_h^1$.
Plugging the exact solution $(\bm v, q)$ into the MAC scheme, 
\begin{align}
    \wt E_h^d(\bm x) = E_h^d(\bm x) \mp  \chi_{ \mc I_h^d} L_h^d[\wh{\bm v}\chi_{\Omega_\mp}, \wh q\chi_{\Omega_\mp}](\bm x), \quad\bm x\in\Omega_\pm \cap \mc T_h^1,
\end{align}
where $E_h^d(\bm x) = \mc O(h^2), d=1,2,3$ is the small local truncation error that will not affect the accuracy, and the remaining term is large and should be incorporated into the scheme as a correction term.
Let $(\wh{\bm v}_h, \wh{q}_h)$ be the numerical approximation of $(\wh{\bm v}, \wh q)$; we write the corrected MAC scheme as
\begin{align}
    &L_h^d[\bm v_h, q_h](x_i,y_j) = f_{i,j}^{(d)}  \mp  \chi_{  \mc I_h^d} L_h^d[\wh{\bm v}_h\chi_{\Omega_\mp}, \wh q_h\chi_{\Omega_\mp}](x_i,y_j) \\
    &=f_{i,j}^{(d)}  - \chi_{ \mc I_h^d\cap\Omega_+} L_h^d[\wh{\bm v}_h\chi_{\Omega_-}, \wh q_h\chi_{\Omega_-}](x_i,y_j) +  \chi_{ \Omega_-\cap \mc I_h^d} L_h^d[\wh{\bm v}_h\chi_{\Omega_+}, \wh q_h\chi_{\Omega_+}](x_i,y_j) .
\end{align}
The local truncation error of the corrected scheme is the sum of $E_h^d$ and the error due to the numerical approximation of the correction functions. As long as $(\wh{\bm v}_h, \wh{q}_h)$ is sufficiently accurate such that the introduced local truncation error at irregular grid nodes is one order lower than that at regular grid nodes, one can expect that the scheme can retain its convergence rate, since $\Gamma$ is a codimension one object and irregular grid nodes are far fewer~\cite{Beale2006,Dong2023}.

By solving the corrected MAC scheme, one obtains the grid function $(\bm v_h, q_h)$, with which we can extract the integrals by interpolating the limiting values of the potential function at the interface; see~\cite{Ying2007,ZhouYing2024Correction} for details. 

\subsection{Geometric multigrid solver}
In matrix form, the corrected MAC scheme can be written as $\mathbf{L}_h\mathbf{v}_h = \mathbf{b}$, where 
\begin{equation}
    \mathbf{L}_h = 
    \begin{bmatrix}
        \mathbf{A}_h - \kappa\mathbf{I}_h & \mathbf{B}^T_h\\
        \mathbf{B}_h & \mathbf{0}
    \end{bmatrix},\quad 
    \mathbf{v}_h = 
    \begin{bmatrix}
        \mathbf{u}_h\\
        \bm p_h
    \end{bmatrix}, \quad 
    \mathbf{b}_h = 
    \begin{bmatrix}
        \mathbf{f}_h\\
        \bm g_h
    \end{bmatrix}.
\end{equation}
Due to the corrections, the right-hand side $\bm g_h$ is generally nonzero and must be adjusted by subtracting a constant so that $\bm g_h$ has zero mean, ensuring that the right-hand side lies in the range of $\bm L_h$. 
The resulting saddle-point system is solved using a geometric multigrid method with a distributive Gauss-Seidel smoother.
For a diagonally-dominant matrix $\mathbf{A} = \mathbf{D}-\mathbf{L}-\mathbf{U}$ with diagonal part $\mathbf{D}$, lower triangular part $\mathbf{L}$, and upper triangular part $\mathbf{U}$, the Gauss-Seidel approximation to the inverse $\mathbf{A}^{-1}$ is defined as $(\mathbf{A}^{-1})_{GS}=(\mathbf{D}-\mathbf{L})^{-1}$.
Thus, the distributive Gauss-Seidel method is given as
\begin{equation}
    \mathbf{v}_h^{m+1} = 
    \mathbf{v}_h^{m} + (\mathbf{L}_h^{-1})_{DGS} \left (
    \mathbf{b}_h -
    \mathbf{L}_h
    \mathbf{v}_h^m
    \right ),
\end{equation}
where $(\mathbf{L}_h^{-1})_{DGS}$ is the distributive Gauss-Seidel approximation to $\mathbf{L}_h^{-1}$, given by
\begin{equation}
    (\mathbf{L}_h^{-1})_{DGS} = 
    \begin{bmatrix}
        \mathbf{I}_h & \mathbf{B}_h^T \\
        \mathbf{0} & \kappa\mathbf{I}_h-\mathbf{B}_h\mathbf{B}_h^T
    \end{bmatrix}
    \begin{bmatrix}
        ((\mathbf{A}_h - \kappa\mathbf{I}_h)^{-1})_{GS} & \mathbf{0}\\
        -(\mathbf{D}_h^{-1})_{GS}\mathbf{B}_h((\mathbf{A}_h - \kappa\mathbf{I}_h)^{-1})_{GS} & (\mathbf{D}_h^{-1})_{GS}
    \end{bmatrix}.
\end{equation}
The prolongation operator $\mathcal{I}_{2h}^h$ for transferring a coarse-grid function to a fine-grid function is defined as bi-/trilinear interpolation for the velocity components and piecewise constant interpolation for the pressure.
The restriction operator $\mathcal{I}_{h}^{2h}$ for transferring a fine grid function to a coarse grid function is defined as the transpose of the prolongation operator, i.e., $\mathcal{I}_{h}^{2h} = (\mathcal{I}_{2h}^h)^T$.
As indicated in~\cite{trottenberg2001multigrid}, the mean-zero constraint can be imposed on the coarsest grid and will not affect convergence of the iteration.
To enforce the mean-zero constraint on $p_h$, we subtract its mean value after solving the linear system.

\section{Numerical method for local Cauchy problem}\label{sec5}
\subsection{Partition of unity collocation method}
To solve the local Cauchy problem in the narrow band $\Omega_\Gamma$, we use a partition of unity to represent the numerical solution.
Let $\mc O_i$ be balls centered at an interface point $\bm X_i$ with a radius $R$. For sufficiently large $R$ (normally of order $\mc O(h)$ where $h$ is the bulk grid spacing) such that $\{\mc O_i\}_{i=0}^{M-1}$ form an overlapping cover of $\Omega_\Gamma \subset \cup_{i=0}^{M-1}\mc O_i$.
Let $\rho_i(\bm x)$ be a partition of unity function subordinate to the set $\mc O_i$ for $i=0,1,\cdots, M-1$ that satisfies
\begin{equation}
    \text{supp }\rho_i \subset \overline{\mc O_i}, \quad \sum_{i=0}^{M-1} \rho_i (\bm x) \equiv 1,\quad \bm x\in\Omega_\Gamma.
\end{equation}
We represent any function $f$ defined on $\Omega_\Gamma$ by
\begin{equation}
     f(\bm x) = \sum_{i = 0}^{M - 1} \rho_i(\bm x) f_i(\bm x),
\end{equation}
where $f_i$ is the local function with compact support in $\overline{\mc O_i}$.
Note that we do not specify any smoothness on the weight function of the partition of unity, as the correction function is only used to modify the local truncation error. In practice, a simple and discontinuous $\rho_i$ is given by
\begin{equation}
    \rho_i(\bm x) = 
    \begin{cases}
        1 & \text{ if } i = \arg\min_j\norm{\bm x - \bm X_j}, \\
        0 & \text{ otherwise}.
    \end{cases}
\end{equation}
Once we can determine the numerical solution of the local Cauchy problem in each $\mc O_i$, we can use the partition of unity function to patch together the local solutions.

Let $(\wh{\bm v}_h, \wh{q}_h)$ be the numerical solution to the local Cauchy problem, expressed as
\begin{equation}
    \wh{\bm v}_h(\bm x) = \sum_{i=0}^{M-1} \rho_i(\bm x) \wh{\bm v}_{h, i}(\bm x), \quad \wh{q}_h(\bm x) = \sum_{i=0}^{M-1} \rho_i(\bm x) \wh{q}_{h, i}(\bm x),
\end{equation}
where $(\wh{\bm v}_{h,i}, \wh{q}_{h,i})$ is the local solution in the ball $\mc O_i$. Treating the normal direction as the ``time'' direction, we numerically approximate the local Cauchy problem~\eqref{eqn:cau-pro} with an explicit time-stepping method, leading to decoupled subproblems for each local solution $(\wh{\bm v}_{h,i}, \wh{q}_{h,i})$, which serve as numerical solutions to the following localized problem in the small domain $\mc O_i$:
\begin{subequations}
\begin{align}
    \Delta  \wh{\bm v}_i - c^2 \wh{\bm v}_i - \grad \wh q_i &= \wh{\bm F},\quad\grad \cdot \wh{\bm v}_i = 0,\quad \text{ in }\mc O_i, \\
    \wh{\bm v}_i &= \bm \Phi, \quad  \wh{\bm \sigma}_i \bm n = \bm \Psi, \quad \text{ on }\Gamma\cap\mc O_i.
\end{align}    
\end{subequations}

Assume that each component of $(\wh{\bm v}_{h,i} \text{ and } \wh{q}_{h,i})$ is represented as linear combinations of multivariate polynomials 
\begin{equation}
    \wh{\bm v}_{h,i}(\bm x) = \sum_{j = 1}^{6} \bm c_{ij} \phi_j((\bm x - \bm X_i)/R), \quad \wh{q}_{h,i}(x, y) = \sum_{j=1}^{3}\frac{ d_{ij}}{R} \phi_j((\bm x - \bm X_i)/R),
\end{equation}
where $\phi_j$ is chosen as a polynomial of degree $\leq m$,
\begin{equation}
    \phi_j(\bm x) \in\Pi_m = \left\{\sum_{|\alpha|\leq m }b_\alpha \bm x^\alpha \right\},\quad\bm x^\alpha = x_1^{\alpha_1} x_2^{\alpha_2}, \quad \alpha = (\alpha_1, \alpha_2).
\end{equation}
The simplest way is to choose $\phi_j$ as the monomial basis $\bm x^\alpha$, so that
\begin{equation}
\{\phi_j(x,y)\} = \{1, x, y, x^2, y^2, xy\}.
\end{equation}
Since in the problem \eqref{eqn:cau-pro}, the velocity has one more derivative than the pressure, we set $\wh{\bm v}_{h,i}$ as a polynomial of degree $2$ and $\wh{q}_{h,i}$ as a polynomial of degree $1$, similar to the well-known Taylor-Hood element for incompressible Stokes equations used in the finite element framework. Since 
\begin{equation}
    \text{dim }\Pi_m = \frac{1}{2}(m+1)(m+2),
\end{equation}
we have $M_u = 6$ and $M_p = 3$.

The coefficients $(\bm c_{ij}, d_{ij})$ can be determined by collocation.
Let $ \{\bm q^{(1)}_k\}_{k=1}^{K^{(1)}}\subset \mc O_i\cap\Gamma$ be distinct collocation points for the Dirichlet boundary condition; let $\{\bm q^{(2)}_k\}_{k=1}^{K^{(2)}}\subset \mc O_i\cap\Gamma$ be distinct collocation points for the Neumann boundary condition; and let $\{\bm q^{(3)}_k\}_{k=1}^{K^{(3)}}$ and $\{\bm q^{(4)}_k\}_{k=1}^{K^{(4)}}\subset \mc O_i$ be distinct collocation points for the momentum equation and the divergence-free condition, respectively. We arrive at the following collocation problem
\begin{subequations}
    \begin{align}
        \sum_j \phi_j(\bm q_k^{(1)} / R) \bm c_j &= \bm \Phi(\bm q_k^{(1)}),\\
        \sum_j \paren{\bm c_j\otimes\grad \phi_j + \grad\phi_j \otimes \bm c_j - d_j\phi_j \bm I}(\bm q_k^{(2)}/R)\bm n(\bm q_k^{(2)}) &= R\bm\Psi(\bm q_k^{(2)}),\\
        \sum_j\paren{(\Delta-c^2)\phi_j \bm c_j - \grad\phi_j d_j}(\bm q^{(3)}/R) &= R^2 \bm F(\bm q^{(3)}),\\
        \sum_j\paren{\partial_x \phi_j c_j^{(1)} + \partial_y \phi_j c_j^{(2)}}(\bm q_k^{(4)}/R) & = 0.
    \end{align}
\end{subequations}
We abbreviate it in matrix form
\begin{equation}\label{eqn:Mcb}
    \bm M \bm c = \bm b,
\end{equation}
where $\bm c \in \R ^{15}$ consists of the unknowns $\bm c_{ij}, d_{ij}$.

The selection of collocation points is essential for the performance of the collocation. 
A first attempt is to choose more collocation points than needed (more than the number of unknowns) so that $\bm M$ has full rank and then solve the resulting linear system in the least squares sense. In this way, since each equation is not satisfied exactly, it is useful to consider a weighted least squares method to balance the accuracy for different collocation equations. 
Noticing that the PDE and boundary conditions involve derivatives of different orders, the spatial parameter $R$ provides a natural choice for the weight so that the matrix $\bm M$ is not affected by $R$. 
We multiply both sides of the collocation equations by a power of $R$ based on the spatial scaling of the equation, which is essentially related to the order of the derivatives involved. This serves as a diagonal preconditioner that can significantly reduce the effect of the parameter $R$ on the condition number of $\bm M$, which is around $10^2\sim 10^3$ regardless of how small $R$ is.
In fact, if $\Gamma$ is a straight line, then the matrix $\bm M$ is independent of $R$.


\subsection{Minimal choice for collocation points}
The least-squares method for solving the collocation problem is robust; however, it typically requires more collocation points than necessary to ensure that the system has full rank, which increases the computational overhead. To optimize the selection of collocation points, we propose a strategy that uses the minimum number of points so that the resulting linear system is square and solvable.

For convenience, we suppose that, under translation and rotation, the local parameterization of $\Gamma\cap\mc O_i$ is given by $\bm X(0) = \bm X_i = \bm 0$ and $\partial_s \bm X(0) = (1,0)^T$, where $s$ is the arc-length parameter. 
For each $d=1,2,3$, let $\eta_k^{(d)}$ be distinct constants independent of $R$ such that $\bm X(\eta_k^{(d)}R)\in\Gamma\cap\mc O_i$.
Let $\bm q_k^{(1)} = \bm X(\eta_k^{(1)} R)$ for $k = 1,2,3$ be distinct collocation points for the Dirichlet boundary condition. 
Let $\bm q_k^{(2)} = \bm X(\eta_k^{(2)} R)$ for $k = 1,2$ be distinct collocation points for the Neumann boundary condition. 
Let $\bm q^{(3)} \in \Gamma\cap\mc O_i$ be a collocation point for the Brinkman equation and let $\bm q_k^{(4)} \in \Gamma\cap\mc O_i$ for $k=1,2,3$ be collocation points for the divergence free condition. We choose $\bm q_k^{(4)}$ so that they are not on the same line.
A simple count shows that the number of collocation points is equal to the number of unknowns, both $15$.
Hence, the solvability of the linear system~\eqref{eqn:Mcb} is equivalent to the uniqueness due to the Fredholm alternative.
We now analyze the solvability of the system~\eqref{eqn:Mcb}.

We start by discussing the case where $\Gamma\cap \mc O_i$ is locally a straight line:
\begin{lemma}
    Let $\Gamma\cap \mc O_i$ be a straight line segment. Then $\bm M$ is independent of $R$ and is invertible.
\end{lemma}
\begin{proof}
We can write $\Gamma\cap \mc O_i$ as
\begin{equation}
    \Gamma\cap \mc O_i = \{(x,y):x\in (-R, R), y = 0\}.
\end{equation}
Then, the collocation points are given by $(\eta_k^{(1)} R, 0)\in \Gamma\cap \mc O_i$ with $k = 1,2,3$ for the Dirichlet boundary condition, $(\eta_k^{(2)} R, 0)\in \Gamma\cap \mc O_i$ with $k = 1,2$ for the Neumann boundary condition. The collocation point for the Brinkman equation is chosen as $(0, 0)\in \mc O_i$ and the divergence free condition $(\eta_1^{(4)}, 0), (\eta_2^{(4)}, 0), (0, \eta_3^{(4)})$ with $\eta_3^{(4)}\neq 0$.

To show that $\bm M$ is invertible, we only need to demonstrate that the homogeneous system $\bm M \bm c = \bm 0$ has only a zero solution.
The Dirichlet boundary condition simplifies to
\begin{equation}
    \bm c_1 + \eta_k^{(1)} \bm c_2 +  (\eta_k^{(1)} )^2\bm c_4 = \bm 0, \quad k = 1,2,3.
\end{equation}
Since $\eta_k^{(1)}$ are distinct, this is equivalent to the polynomial interpolation in 1D and the coefficient matrix is the Vandermonde matrix, which is invertible. Then, we have $\bm c_1 = \bm c_2 = \bm c_4 = \bm 0$. 
The collocation equation for the Neumann boundary conditions and the equation for the divergence-free condition at $\bm q_1^{(4)}, \bm q_2^{(4)}$ lead to
\begin{equation}
    \begin{aligned}
        \bm c_3 + \eta_k^{(2)}\bm c_6 + c_3^{(2)} 
        \begin{bmatrix}
            0\\
            1
        \end{bmatrix}
        + c_6^{(2)}
        \begin{bmatrix}
            0\\
            \eta_k^{(2)}
        \end{bmatrix}
        - d_1
        \begin{bmatrix}
            0\\
            1
        \end{bmatrix}
        -d_2
        \begin{bmatrix}
            0\\
            \eta_k^{(2)}
        \end{bmatrix} &= \bm 0, \\
    c_3^{(2)} + \eta_k^{(2)} c_6^{(2)} &= 0.
    \end{aligned}
\end{equation}
Since $\eta_1^{(2)} \neq \eta_2^{(2)}$ and $\eta_1^{(4)} \neq \eta_2^{(4)}$, we obtain
\begin{equation}
    \bm c_3 = \bm c_4=\bm 0,\quad d_1 = d_2 = 0.
\end{equation}
Considering the equation at $\bm q^{(3)}$ and $\bm q_3^{(4)}$, we have
\begin{equation}
\begin{aligned}
    2\bm c_5 - d_3 
    \begin{bmatrix}
        0 \\
        1
    \end{bmatrix} = \bm 0,\quad
    2\eta_3^{(4)} c_5^{(2)} = 0.    
\end{aligned}
\end{equation}
Since $\eta_3^{(4)}\neq 0$. We conclude $\bm c_5 = \bm 0$ and $d_3 = 0$. 
Hence, we have zero solution $\bm c =\bm 0$.
The above calculation also shows that $\bm M$ is independent of $R$.
\end{proof}

Now we consider the curved-interface case, in which the collocation points for the boundary conditions deviate from the line $y=0$ and therefore differ from those in the straight-line case. The collocation points for the PDEs can still be chosen to be the same. Therefore, only the entries in the coefficient matrix $\bm M$ related to the boundary conditions change.

\begin{proposition}
    Let $\Gamma$ be a $C^2$ curve. There exists $R_0 > 0$ such that for all $R \leq R_0$, $\bm M$ is invertible and the inverse $\bm M^{-1}$ is uniformly bounded in $R$.
\end{proposition}
\begin{proof}
Let $\wh{\bm X}(s) = (s,0)$ be the straight line and $\wh{\bm M}$ be the coefficient matrix associated with the collocation problem with the boundary $\wh{\bm X}$. 
Then, using $\wh {\bm M}$ as a preconditioner, we can rewrite the system as follows
\begin{equation}
    \paren{\bm I + \wh{\bm M}^{-1}(\bm M - \wh{\bm M})}\bm c = \wh{\bm M}^{-1}\bm b.
\end{equation}
We need to show that the coefficient matrix is a small perturbation of an identity.
Recall that $\bm X(0) = 0$ and $\partial_s\bm X(0) = (1,0)$, for any point $\bm q_k^{(d)} = \bm X(\eta_k^{(d)}R)$, we have the estimate
\begin{equation}
\begin{aligned}
    |\bm q_k^{(d)} - \wh{\bm X}(\eta_k^{(d)}R)| &= |\bm X(0) + \eta_k^{(d)}R\partial_s\bm X(0) + \frac{(\eta_k^{(d)}R)^2}{2}\bm X(\xi) -\wh{\bm X}(\eta_k^{(d)}R)| \\
    &= |\frac{(\eta_k^{(d)}R)^2}{2}\bm X(\xi)| \leq \frac{1}{2} R^2 \norm{\partial_s^2\bm X}_\infty.
\end{aligned}
\end{equation}
For every $\alpha, \beta \in \mathbb N$, we have
\begin{equation}
\begin{aligned}    
    |\partial_x^\alpha\partial_y^\beta \phi_j (\bm q_k^{(d)}/R) -& \partial_x^\alpha\partial_y^\beta \phi_j (\wh{\bm X}(\eta_k^{(d)}R)/R) | =\frac{1}{R} |\bm q_k^{(d)} - \wh{\bm X}(\eta_k^{(d)}R)||\partial_x^\alpha\partial_y^\beta \phi_j(\bm \xi)| \\
    &\leq \frac{R}{2} \sup_{|\bm \xi|\leq 1}|\partial_x^\alpha\partial_y^\beta \phi_j(\bm \xi)| \norm{\partial_s^2\bm X}_\infty \leq C R ,
\end{aligned}
\end{equation}
where the constant $C$ depends on $\bm X$ but is independent of $R$.
Since $\phi_j$ is a polynomial, the term $\sup_{|\bm \xi|\leq 1}|\partial_x^\alpha\partial_y^\beta \phi_j(\bm \xi)|$ is bounded.
Hence, we have the estimate for the difference matrix $\bm M - \wh{\bm M}$,
\begin{equation}
    \norm{\bm M - \wh{\bm M}}_{\infty} \leq \max_{1\leq i\leq 15}\sum_{j=1}^{15} |\bm M_{ij} - \wh{\bm M}_{ij}| \leq C R .
\end{equation}
Therefore, the perturbation matrix can be bounded as follows
\begin{equation}
    \norm{\wh{\bm M}^{-1}(\bm M - \wh{\bm M})}_\infty \leq \norm{\wh{\bm M}^{-1}}_\infty\norm{\bm M - \wh{\bm M}}_\infty \leq C R.
\end{equation}
Letting $R_0 = 1/(2C)$, then for every $R<R_0$ the perturbation matrix is strictly bounded by $1$ and we have 
\begin{equation}
\begin{aligned}
    \norm{\bm M^{-1}}_\infty & \leq \norm{(\bm I + \wh{\bm M}^{-1}(\bm M-\wh{\bm M}))^{-1}}_\infty \norm{\wh{\bm M}^{-1}}_\infty 
    \\& \leq \frac{\norm{\wh{\bm M}^{-1}}_\infty}{1 - \norm{\wh{\bm M}^{-1}(\bm M - \wh{\bm M})}_\infty} \leq 2 \norm{\wh{\bm M}^{-1}}_\infty.
\end{aligned}
\end{equation}
The bound for $\norm{\bm M^{-1}}_\infty$ is uniform in $R$ since $\wh{\bm M}$ is independent of $R$.
\end{proof}

The constant $C$ that emerged in the proof depends on $\partial_s^2\bm X = -\kappa\bm n$, where $\kappa$ represents the curvature of $\Gamma$. This implies that for curves exhibiting large curvature, particularly those that are non-smooth or highly oscillatory, it is crucial to choose $R$ to be quite small to ensure the collocation problem can be solved.

\section{Numerical examples}\label{sec6}
In this section, we present numerical results for the proposed method applied to the interface problem~\eqref{eqn:stokes-ip} in order to assess its accuracy and efficiency. 
We first solve the interface problem for a range of parameter values and report the accuracy and computational cost of the algorithm. 
We then apply the method to several time-dependent moving interface problems.

The errors $(\bm e_{u}, e_p) = (\bm u-\bm u_h, p - p_h)$ are measured in the discrete $\ell^2$ norms
\begin{equation}
    \norm{\bm e_{u}}^2 
    = h^2\paren{\sum_{i=1}^{N-1}\sum_{j=1}^N \bigl(e_{u,ij}^{(1)}\bigr)^2 
      + \sum_{i=1}^N\sum_{j=1}^{N-1}\bigl(e_{u,ij}^{(2)}\bigr)^2},
    \quad 
    \norm{e_p}^2 
    = h^2 \sum_{i=1}^N\sum_{j=1}^N e_{p,ij}^2.
\end{equation}
We also consider the discrete $\ell^2$-norm of the gradient,
\begin{equation}
\begin{aligned}
    \norm{\nabla_h \bm e_u}^2 
    &= h^2 \paren{\sum_{i=1}^N\sum_{j=1}^{N} \bigl(\delta_{h,1} e_{u,ij}^{(1)}\bigr)^2 
      + \sum_{i=1}^{N-1}\sum_{j=0}^N \rho_j^y \bigl(\delta_{h,2}^- e_{u,ij}^{(1)}\bigr)^2 } \\
    &\quad + h^2\paren{\sum_{i=0}^N \sum_{j=1}^{N-1} \rho_i^x \bigl(\delta_{h,1}^- e_{u,ij}^{(2)}\bigr)^2  
      + \sum_{i=1}^N\sum_{j=1}^N \bigl(\delta_{h,2}^- e_{u,ij}^{(2)}\bigr)^2 }.
\end{aligned}
\end{equation}

In all numerical experiments, the GMRES solver is initialized with the zero vector, and the iteration is terminated when the relative residual satisfies the tolerance $\texttt{tol} = 10^{-8}$. 
All computations are carried out on a desktop computer equipped with an 11th\,Gen Intel\textsuperscript{\textregistered} Core\texttrademark{} i9-11900H CPU (2.50\,GHz) and 15\,GiB of RAM. 
The code is implemented in \texttt{C++} and compiled with \texttt{GCC}~11.4.0.

\subsection{Example 1: Static interface problem with $\kappa_+/\mu_+ = \kappa_-/\mu_-$}
In the first example, we solve the interface problem~\eqref{eqn:stokes-ip} with equal coefficient ratios $\kappa_+/\mu_+ = \kappa_-/\mu_-$ and consider the boundary integral equation~\eqref{eqn:bie-1}. 
To examine the performance of the method for different parameter choices, we fix $\mu_- = 1$ and vary the ratios $\omega = \mu_+ / \mu_-$ and $\eta = \kappa_+ / \mu_+ =  \kappa_- / \mu_-$. 
The computational domain is $\mc B = (-2,2)^2$, discretized by a uniform Cartesian grid, and the interface is taken to be the unit circle centered at the origin.

We employ a manufactured solution to assess the numerical error. 
The right-hand side, interface conditions, and boundary condition on the outer boundary are chosen so that the exact solution is
\begin{equation}
\begin{aligned}
    u^{(1)}(x,y) &= 
    \begin{cases}
        \dfrac{y}{4}\bigl(x^2 + y^2\bigr), & (x,y)\in\Omega_+,\\[4pt]
        \dfrac{y}{r} - \dfrac{3}{4}y, & (x,y)\in\Omega_-,
    \end{cases} \\
    u^{(2)}(x,y) &=
    \begin{cases}
        -\dfrac{1}{4} x y^2, & (x,y)\in\Omega_+,\\[4pt]
        -\dfrac{x}{r} + \dfrac{x}{4}\bigl(3 + x^2\bigr), & (x,y)\in\Omega_-,
    \end{cases} \\
    p(x,y) &=
    \begin{cases}
        5, & (x,y)\in\Omega_+,\\[4pt]
        y\Bigl(-\dfrac{3}{4}x^3 + \dfrac{3}{8}x\Bigr), & (x,y)\in\Omega_-,
    \end{cases}
\end{aligned}
\end{equation}
where $r = \sqrt{x^2 + y^2}$.

For fixed $\eta = 1$, we vary $\omega = 10^k$ for $k = -3,-1,1,3$. 
We also consider the case $\omega = 10$ with $\eta = 10^3$.
In all cases, the numerical results show second-order convergence in the discrete $\ell^2$ norm for the velocity, its discrete gradient, and the pressure. 
Both the velocity and pressure errors are essentially unaffected by the parameters $\omega$ and $\eta$; see Tables~\ref{tab:eg1-1}-\ref{tab:eg1-5}. 
Moreover, for all parameter choices, the GMRES iteration count remains stable as the mesh is refined, and on finer meshes it even decreases slightly, reflecting that the finer discretization better captures the well-conditioned nature of~\eqref{eqn:bie-1}.

\begin{table}[htbp]
\centering
\caption{Numerical results for the Example 1: $\omega = 0.001, \eta = 1$.}
\label{tab:eg1-1}
\begin{tabular}{r r r r r r r r}
\hline
$N$ & $\norm{\bm e_u}$ & Order & $\norm{\nabla_h \bm e_u}$ & Order & $\norm{\bm e_p}$ & Order & \#GMRES \\
\hline
  64  & 6.32e-04 & -    & 1.19e-03 & - & 1.30e-03 & - & 9 \\
 128  & 1.55e-04 & 2.03 & 3.19e-04 & 1.90 & 3.75e-04 & 1.79 & 9 \\
 256  & 3.89e-05 & 1.99 & 8.52e-05 & 1.91 & 1.05e-04 & 1.83 & 8 \\
 512  & 9.23e-06 & 2.07 & 2.24e-05 & 1.92 & 2.93e-05 & 1.85 & 8 \\
1024  & 2.32e-06 & 1.99 & 5.93e-06 & 1.92 & 7.95e-06 & 1.88 & 7 \\
\hline
\end{tabular}
\end{table}

\begin{table}[htbp]
\centering
\caption{Numerical results for the Example 1: $\omega = 0.1, \eta = 1$.}
\label{tab:eg1-2}
\begin{tabular}{r r r r r r r r}
\hline
$N$ & $\norm{\bm e_u}$ & Order & $\norm{\nabla_h \bm e_u}$ & Order & $\norm{\bm e_p}$ & Order & \#GMRES \\
\hline
  64  & 6.23e-04 & - & 1.18e-03 & - & 1.30e-03 & - & 9 \\
 128  & 1.52e-04 & 2.03 & 3.17e-04 & 1.90 & 3.76e-04 & 1.79 & 8 \\
 256  & 3.82e-05 & 2.00 & 8.48e-05 & 1.90 & 1.06e-04 & 1.83 & 8 \\
 512  & 9.13e-06 & 2.06 & 2.24e-05 & 1.92 & 2.95e-05 & 1.85 & 7 \\
1024  & 2.29e-06 & 1.99 & 5.92e-06 & 1.92 & 7.97e-06 & 1.89 & 6 \\
\hline
\end{tabular}
\end{table}

\begin{table}[htbp]
\centering
\caption{Numerical results for the Example 1: $\omega = 10, \eta = 1$.}
\label{tab:eg1-3}
\begin{tabular}{r r r r r r r r}
\hline
$N$ & $\norm{\bm e_u}$ & Order & $\norm{\nabla_h \bm e_u}$ & Order & $\norm{\bm e_p}$ & Order & \#GMRES \\
\hline
  64  & 4.79e-04 & - & 1.37e-03 & - & 2.50e-03 & - & 10 \\
 128  & 1.18e-04 & 2.02 & 3.61e-04 & 1.92 & 6.32e-04 & 1.98 & 10 \\
 256  & 3.10e-05 & 1.93 & 9.73e-05 & 1.89 & 1.69e-04 & 1.90 &  9 \\
 512  & 7.45e-06 & 2.05 & 2.34e-05 & 2.06 & 4.01e-05 & 2.08 &  8 \\
1024  & 1.94e-06 & 1.94 & 6.15e-06 & 1.93 & 1.05e-05 & 1.92 &  6 \\
\hline
\end{tabular}
\end{table}

\begin{table}[htbp]
\centering
\caption{Numerical results for the Example 1: $\omega = 1000, \eta = 1$.}
\label{tab:eg1-4}
\begin{tabular}{r r r r r r r r}
\hline
$N$ & $\norm{\bm e_u}$ & Order & $\norm{\nabla_h \bm e_u}$ & Order & $\norm{\bm e_p}$ & Order & \#GMRES \\
\hline
  64  & 8.33e-04 & - & 1.84e-03 & - & 3.35e-03 & - & 13 \\
 128  & 2.11e-04 & 1.98 & 4.76e-04 & 1.95 & 8.65e-04 & 1.96 & 13 \\
 256  & 5.98e-05 & 1.82 & 1.30e-04 & 1.87 & 2.16e-04 & 2.00 & 11 \\
 512  & 7.56e-06 & 2.98 & 2.53e-05 & 2.36 & 4.80e-05 & 2.17 & 10 \\
1024  & 1.94e-06 & 1.96 & 6.55e-06 & 1.95 & 1.25e-05 & 1.94 &  9 \\
\hline
\end{tabular}
\end{table}

\begin{table}[htbp]
\centering
\caption{Numerical results for the Example 1: $\omega = 10, \eta = 1000$.}
\label{tab:eg1-5}
\begin{tabular}{r r r r r r r r}
\hline
$N$ & $\norm{\bm e_u}$ & Order & $\norm{\nabla_h \bm e_u}$ & Order & $\norm{\bm e_p}$ & Order & \#GMRES \\
\hline
  64  & 5.64e-04 & - & 1.47e-03 & - & 2.51e-03 & - & 11 \\
 128  & 1.47e-04 & 1.94 & 3.89e-04 & 1.91 & 6.31e-04 & 1.99 & 11 \\
 256  & 4.25e-05 & 1.79 & 1.07e-04 & 1.86 & 1.69e-04 & 1.90 & 10 \\
 512  & 8.41e-06 & 2.34 & 2.34e-05 & 2.20 & 3.99e-05 & 2.08 &  9 \\
1024  & 2.36e-06 & 1.84 & 6.20e-06 & 1.92 & 1.05e-05 & 1.92 &  6 \\
\hline
\end{tabular}
\end{table}

\subsection{Example 2: Static interface problem with $\kappa_+/\mu_+ \neq \kappa_-/\mu_-$}
In this example, we consider the generic case with different coefficient ratios, 
$\kappa_+/\mu_+ \neq \kappa_-/\mu_-$, and solve the interface problem~\eqref{eqn:stokes-ip}
using the boundary integral equation~\eqref{eqn:bie-2}.
To examine the performance of the method for different parameter choices, we fix $\mu_- = 1$
and vary the parameters $\mu_+$ and $\kappa_+, \kappa_-$.
As in the previous example, we use a manufactured solution to evaluate the numerical error.
The right-hand side, interface conditions, and outer boundary condition are chosen so that the
exact solution satisfies
\begin{equation}
\begin{aligned}
    u^{(1)}(x,y) &= 
    \begin{cases}
        \sin x \cos y, & (x,y)\in\Omega_+,\\[4pt]
        \dfrac{y}{r} - \dfrac{3}{4}y, & (x,y)\in\Omega_-,
    \end{cases} \\
    u^{(2)}(x,y) &=
    \begin{cases}
        -\cos x \sin y, & (x,y)\in\Omega_+,\\[4pt]
        -\dfrac{x}{r} + \dfrac{x}{4}\bigl(3 + x^2\bigr), & (x,y)\in\Omega_-,
    \end{cases} \\
    p(x,y) &=
    \begin{cases}
        \sin x \sin y, & (x,y)\in\Omega_+,\\[4pt]
        \Bigl(-\dfrac{3}{4}x^3 + \dfrac{3}{8}x\Bigr)y, & (x,y)\in\Omega_-,
    \end{cases}
\end{aligned}
\end{equation}
where $r = \sqrt{x^2 + y^2}$.

The computational domain is chosen as $\mc B = (-2,2)^2$.
The interface is given by a six-fold flower,
\begin{equation}
    \bm X(\theta) = \bigl(1 + 0.1 \cos (6\theta)\bigr) 
    \begin{pmatrix}
        \cos \theta\\[2pt]
        \sin \theta
    \end{pmatrix},\quad 
    \theta\in [0, 2\pi).
\end{equation}

The numerical results show that the errors in the velocity, the discrete velocity gradient,
and the pressure all converge at second order in the $\ell^2$ norm.
In contrast to the equal coefficient-ratio case, however, the GMRES iteration count is now
affected by the coefficient ratios: when the interior viscosity $\mu_+$ and the coefficient
$\kappa_+$ are large, more iterations are required for convergence.
This behavior is likely related to the hypersingular terms present in~\eqref{eqn:bie-2},
which make the iterative convergence more sensitive to the coefficients.
Nevertheless, for any fixed choice of $(\mu_+,\kappa_+,\kappa_-)$, the required iteration
numbers are only mildly dependent on the mesh size and typically decrease slightly as the
mesh is refined, indicating an inherent regularizing effect of the discretization that
stabilizes the GMRES iteration.

\begin{table}[htbp]
\centering
\caption{Numerical results for the Example 2: $\mu_+ = 10, \mu_- = 1, \kappa_+ = 1, \kappa_- = 1$.}
\label{tab:eg2-1}
\begin{tabular}{r r r r r r r r}
\hline
$N$ & $\norm{\bm e_u}$ & Order & $\norm{\nabla_h \bm e_u}$ & Order & $\norm{\bm e_p}$ & Order & \#GMRES \\
\hline
  64  & 9.49e-03 & - & 7.13e-03 & - & 1.82e-02 & - & 56 \\
 128  & 9.18e-04 & 3.37 & 6.40e-04 & 3.48 & 1.45e-03 & 3.64 & 45 \\
 256  & 1.90e-04 & 2.28 & 1.19e-04 & 2.43 & 3.50e-04 & 2.05 & 41 \\
 512  & 4.54e-05 & 2.06 & 1.85e-05 & 2.68 & 8.58e-05 & 2.03 & 37 \\
1024  & 1.33e-05 & 1.77 & 3.50e-06 & 2.40 & 2.67e-05 & 1.68 & 38 \\
\hline
\end{tabular}
\end{table}

\begin{table}[htbp]
\centering
\caption{Numerical results for the Example 2: $\mu_+ = 0.1, \mu_- = 1, \kappa_+ = 1, \kappa_- = 1$.}
\label{tab:eg2-2}
\begin{tabular}{r r r r r r r r}
\hline
$N$ & $\norm{\bm e_u}$ & Order & $\norm{\nabla_h \bm e_u}$ & Order & $\norm{\bm e_p}$ & Order & \#GMRES \\
\hline
  64  & 7.85e-04 & - & 1.46e-03 & - & 1.75e-03 & - & 23 \\
 128  & 1.69e-04 & 2.22 & 1.67e-04 & 3.13 & 3.12e-04 & 2.49 & 21 \\
 256  & 5.04e-05 & 1.75 & 2.84e-05 & 2.56 & 6.84e-05 & 2.19 & 19 \\
 512  & 1.01e-05 & 2.31 & 4.06e-06 & 2.81 & 1.36e-05 & 2.33 & 18 \\
1024  & 2.04e-06 & 2.31 & 5.53e-07 & 2.88 & 3.02e-06 & 2.17 & 18 \\
\hline
\end{tabular}
\end{table}

\begin{table}[htbp]
\centering
\caption{Numerical results for the Example 2: $\mu_+ = 0.1, \mu_- = 1, \kappa_+ = 1, \kappa_- = 100$.}
\label{tab:eg2-3}
\begin{tabular}{r r r r r r r r}
\hline
$N$ & $\norm{\bm e_u}$ & Order & $\norm{\nabla_h \bm e_u}$ & Order & $\norm{\bm e_p}$ & Order & \#GMRES \\
\hline
  64  & 2.09e-03 & - & 5.26e-03 & - & 8.33e-03 & - & 21 \\
 128  & 6.08e-04 & 1.78 & 6.99e-04 & 2.91 & 1.25e-03 & 2.74 & 20 \\
 256  & 1.94e-04 & 1.65 & 1.59e-04 & 2.14 & 3.67e-04 & 1.77 & 19 \\
 512  & 3.36e-05 & 2.53 & 2.02e-05 & 2.98 & 6.83e-05 & 2.43 & 19 \\
1024  & 5.95e-06 & 2.49 & 2.87e-06 & 2.82 & 1.44e-05 & 2.24 & 18 \\
\hline
\end{tabular}
\end{table}

\begin{table}[htbp]
\centering
\caption{Numerical results for the Example 2: $\mu_+ = 0.1, \mu_- = 1, \kappa_+ = 100, \kappa_- = 1$.}
\label{tab:eg2-4}
\begin{tabular}{r r r r r r r r}
\hline
$N$ & $\norm{\bm e_u}$ & Order & $\norm{\nabla_h \bm e_u}$ & Order & $\norm{\bm e_p}$ & Order & \#GMRES \\
\hline
  64  & 2.20e-03 & - & 3.76e-03 & - & 3.04e-02 & - & 46 \\
 128  & 2.79e-04 & 2.98 & 6.08e-04 & 2.63 & 1.83e-03 & 4.05 & 40 \\
 256  & 5.87e-05 & 2.25 & 9.79e-05 & 2.63 & 4.77e-04 & 1.94 & 36 \\
 512  & 1.10e-05 & 2.42 & 1.36e-05 & 2.84 & 4.89e-05 & 3.28 & 35 \\
1024  & 3.03e-06 & 1.86 & 2.00e-06 & 2.77 & 1.56e-05 & 1.65 & 35 \\
\hline
\end{tabular}
\end{table}

\subsection{Example 3: moving elastic interface in Brinkman flow}
In this example, we consider a moving interface problem.
Let $\Gamma$ be a moving elastic interface parameterized by 
$\bm X(s, t): \mathbb S\times [0,T]\to \mathbb R^2$, 
where $s$ denotes the Lagrangian coordinate.
At each fixed time $t$, the velocity field is determined by the Brinkman system~\eqref{eqn:stokes-ip}
with jump conditions $\jump{\bm u}  = \bm 0$ and $\jump{\bm\sigma \bm n}  = \bm F$, 
where $\bm F$ is the elastic tension force arising from stretching of $\Gamma$,
\begin{equation}\label{eqn:els-force}
    \bm F(\bm X(s,t), t) 
    = \frac{1}{|\partial_s \bm X|}\,\partial_{s}\!\left(\mc T(|\partial_s \bm X|)\,\frac{\partial_s \bm X}{|\partial_s \bm X|}\right), 
    \qquad 
    \mc T(|\partial_s \bm X|) = T_0\left( |\partial_s \bm X| - \frac{L_0}{2\pi}\right),
\end{equation}
where $T_0$ is a constant depending on the elastic properties of the interface and $L_0$ is the relaxed length of $\Gamma$.
In the following, we set $L_0 = 0$ and $T_0 = 1$.
The coefficients are chosen as $\mu^+ = \kappa^+ = 1$ and $\mu^- = \kappa^- = 0.1$.

Given $\Gamma(t)$, the elastic force $\bm F$ determines the velocity field 
$\bm u = \bm u[\bm X(\cdot, t)]$, and the interface is advected by the flow:
\begin{equation}
    \bm X(s,0) = \bm X_0(s), \qquad 
    \partial_t \bm X(s, t) = \bm u[\bm X(\cdot, t)]\bigl(\bm X(s,t)\bigr), 
    \quad s\in \mathbb S,\; t\in[0,T],
\end{equation}
where $\bm X_0$ is a prescribed initial configuration of $\Gamma$.

The interface $\Gamma$ is discretized by partitioning $[0,2\pi]$ into uniformly spaced points 
$s_i = i\Delta s$, $i=0,1,\ldots,M-1$, with $\Delta s = 2\pi / M$.
Let $\bm X_h^n(\cdot)$ denote the numerical approximation of $\bm X(\cdot, t_n)$ based on periodic cubic splines.
We advance the interface in time using a second-order explicit Runge–Kutta scheme: for $i=0,1,\ldots,M-1$,
\begin{equation}
\begin{aligned}
    \bm X_h^{n+\frac{1}{2}}(s_i) &= \bm X_h^n(s_i) 
        + \frac{\Delta t}{2}\,\bm u_h[\bm X_h^n]\bigl(\bm X_h^n(s_i)\bigr),\\ 
    \bm X_h^{n+1}(s_i) &= \bm X_h^n(s_i) 
        + \Delta t\,\bm u_h[\bm X_h^{n+\frac{1}{2}}]\bigl(\bm X_h^{n+\frac{1}{2}}(s_i)\bigr).
\end{aligned}    
\end{equation}

We study the relaxation of a stretched elastic interface in a viscous Brinkman flow.
The initial shape is given by a six-fold flower,
\begin{equation}
    \bm X_0(\theta) = \bigl(1 + 0.2 \cos (6\theta)\bigr)
    \begin{pmatrix}
        \cos \theta\\[2pt]
        \sin \theta
    \end{pmatrix},\quad 
    \theta\in [0, 2\pi).
\end{equation}
The only driving mechanism is the elastic force generated by the interface.
Due to the incompressibility of the enclosed fluid, the interface evolves at nearly fixed area and relaxes toward a circular equilibrium, as illustrated in Figure~\ref{fig:eg3-fig}.
In the limiting configuration, the pressure exhibits a nonzero jump across the interface to balance the normal elastic force.

The time evolution of the enclosed area and its relative error, together with the GMRES iteration counts and the elastic energy, is shown in Figure~\ref{fig:eg3-res} for meshes with $N=64,128,256$.
On a grid with $N=64$, the enclosed area is preserved with an error of order $10^{-4}$, and this area conservation is further improved on finer meshes.
As the interface relaxes toward a circle, the GMRES iteration count decreases over time.

\begin{figure}[htbp]
    \centering
    \includegraphics[width=0.3\linewidth]{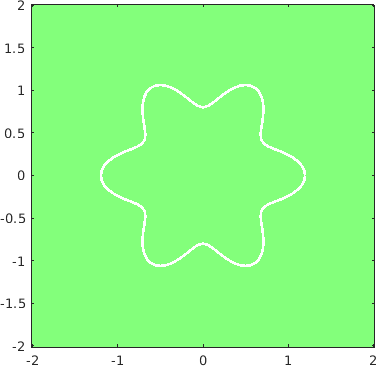}
    \includegraphics[width=0.3\linewidth]{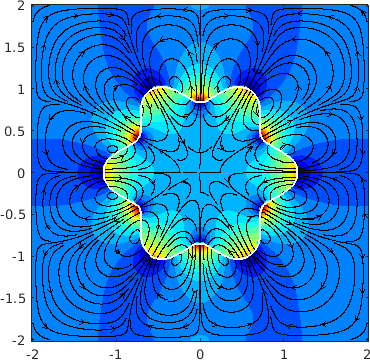}
    \includegraphics[width=0.3\linewidth]{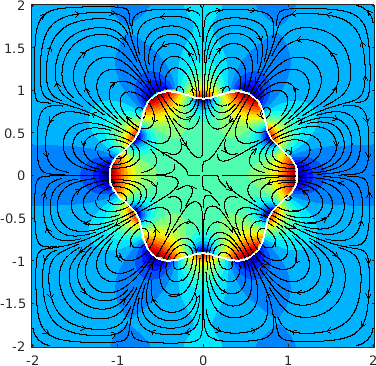}
    \includegraphics[width=0.3\linewidth]{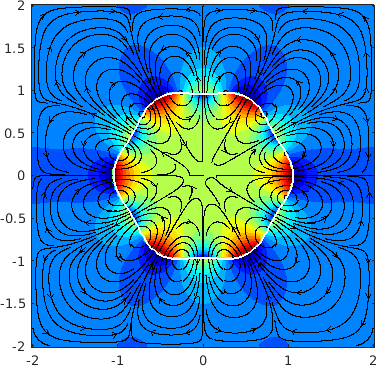}
    \includegraphics[width=0.3\linewidth]{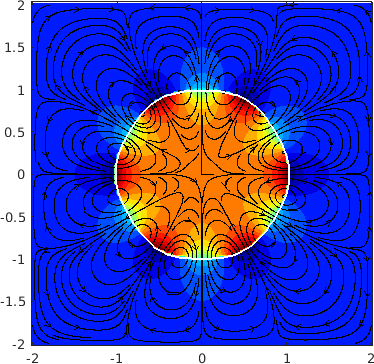}
    \includegraphics[width=0.3\linewidth]{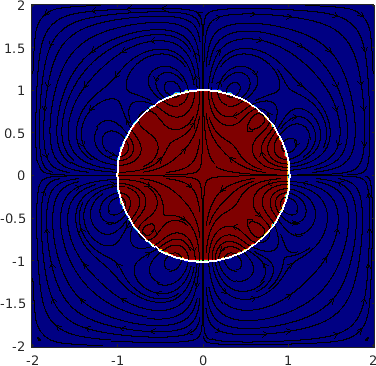}
    \caption{Time evolution of the elastic interface and velocity field in the Brinkman flow. Starting from a highly perturbed shape, the interface relaxes and approaches a circle.}
    \label{fig:eg3-fig}

\end{figure}

\begin{figure}[htbp]
    \centering
    \includegraphics[width=0.3\linewidth]{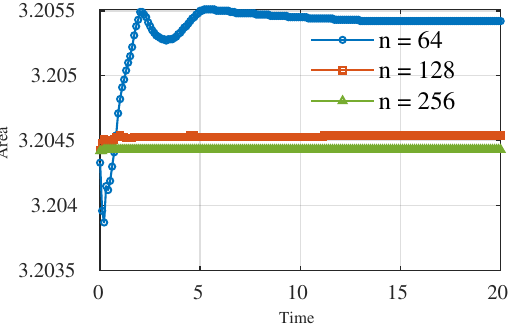}
    \includegraphics[width=0.3\linewidth]{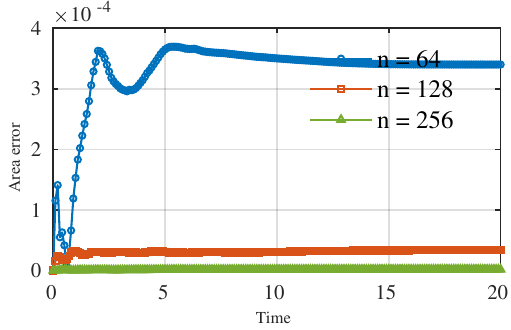}
    \includegraphics[width=0.3\linewidth]{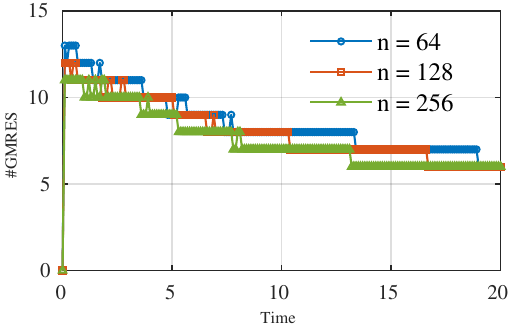}
    \caption{Time evolution of the enclosed area, area error, and GMRES iteration count for $N = 64, 128, 256$. }
    \label{fig:eg3-res}
\end{figure}

\subsection{Example 4: moving elastic interface in unsteady Stokes flow}
In this example, we consider the evolution of an elastic interface in an unsteady Stokes flow.
\begin{align}\label{eqn:unsteady-stokes}
    \partial_t \bm u &= \mu \Delta \bm u - \grad p, 
    \quad \grad\cdot\bm u = 0, 
    \quad \text{in } \mc B \setminus \Gamma,
\end{align}
together with the interface conditions $\jump{\bm u} = 0$ and $\jump{\bm \sigma \bm n} = \bm F$, where $\bm F$ is given in~\eqref{eqn:els-force}.
The viscosities are chosen as $\mu_+ = 1$ and $\mu_- = 0.1$.
The initial interface is given by
\begin{equation}
    \bm X(\theta) = \Big(
        \cos\theta,\,
        \sin\theta\big(1 + 0.4 \sin\theta \sin 3\theta \sin 5\theta\big) 
        + 0.2 \sin(\cos\theta)
    \Big)^T, 
    \quad \theta \in [0,2\pi).
\end{equation}

We discretize~\eqref{eqn:unsteady-stokes} in time using the backward Euler scheme,
\begin{equation}
    \frac{\bm u^{n+1} - \bm u^n}{\Delta t} 
    = \mu \Delta \bm u^{n+1} - \grad p^{n+1},
    \quad \grad\cdot\bm u^{n+1} = 0,
\end{equation}
with interface conditions $\jump{\bm u^{n+1}} = 0$ and 
$\jump{\bm \sigma^{n+1} \bm n^{n+1}} = \bm F^{n+1}$.
This time discretization leads, at each time step, to an elliptic interface problem of the form~\eqref{eqn:stokes-ip} with $\kappa_+ = \kappa_- = 1/\Delta t$, which we solve using the boundary integral formulation~\eqref{eqn:bie-2}.

The time evolution of the interface and the flow field is shown in Figure~\ref{fig:eg4-fig}.
The initially complex interface relaxes due to viscous dissipation and eventually approaches a circular shape.
The time evolution of the enclosed area, the numerical area error, and the GMRES iteration counts are presented in Figure~\ref{fig:eg4-res} for $N = 64, 128, 256$.
The relative area error remains below $10^{-3}$ and decreases under mesh refinement.
The GMRES iteration count decreases from about $30$–$40$ at early times to about $10$–$20$ as the interface becomes nearly circular.

\begin{figure}[htbp]
    \centering
    \includegraphics[width=0.3\linewidth]{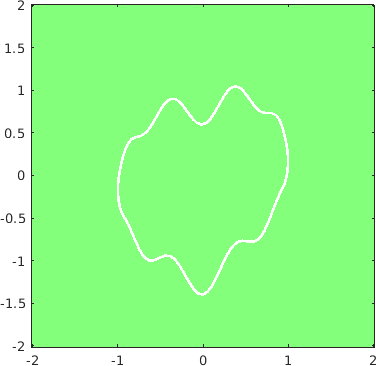}
    \includegraphics[width=0.3\linewidth]{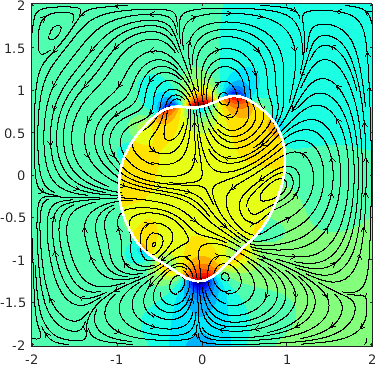}
    \includegraphics[width=0.3\linewidth]{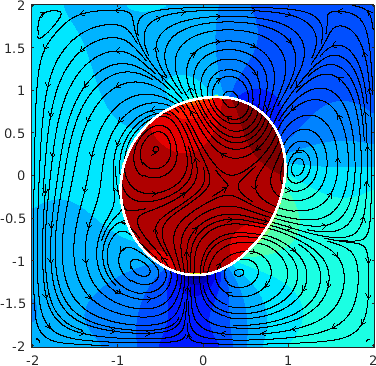}
    \includegraphics[width=0.3\linewidth]{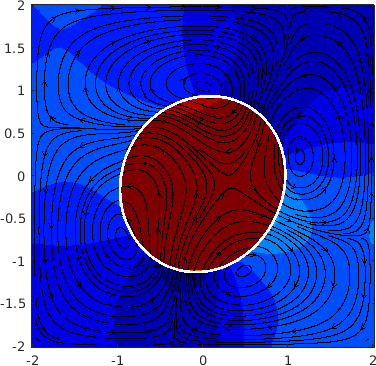}
    \includegraphics[width=0.3\linewidth]{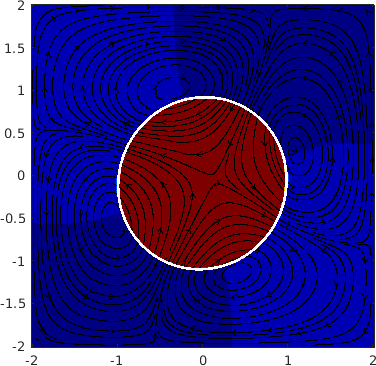}
    \includegraphics[width=0.3\linewidth]{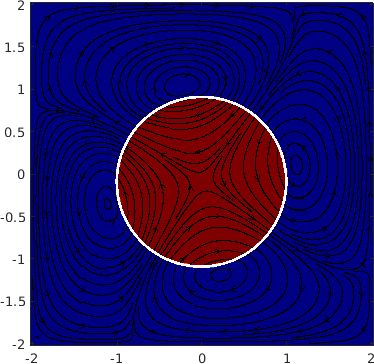}
    \caption{Time evolution of the elastic interface and velocity field in unsteady Stokes flow. Starting from a highly perturbed shape, the interface relaxes and approaches a circle.}
    \label{fig:eg4-fig}
\end{figure}
\begin{figure}[htbp]
    \centering
    \includegraphics[width=0.3\linewidth]{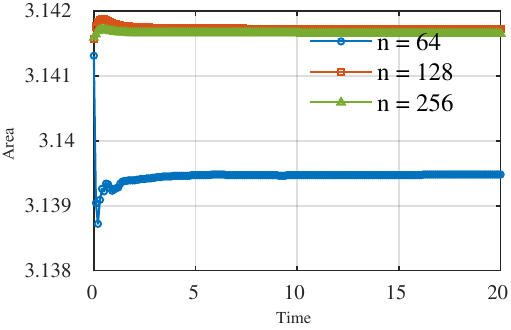}
    \includegraphics[width=0.3\linewidth]{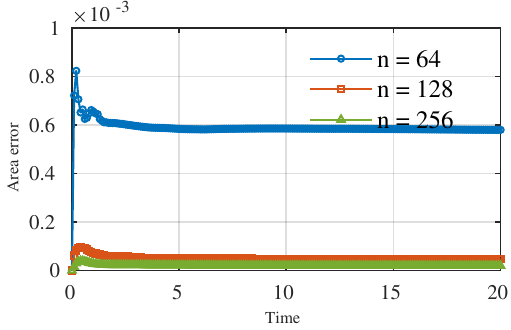}
    \includegraphics[width=0.3\linewidth]{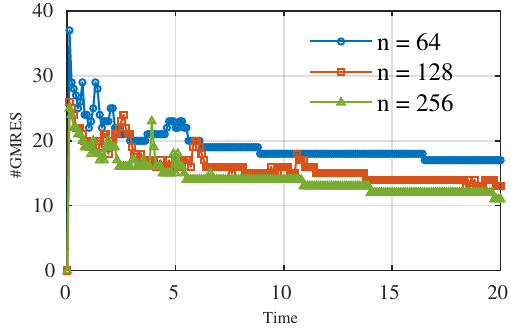}
    \caption{Time evolution of the enclosed area, area error, and GMRES iteration count for $N = 64, 128, 256$. }
    \label{fig:eg4-res}
\end{figure}

\subsection{Example 5: moving elastic interface in shear flow}
In this example, we consider an elastic interface in a Brinkman shear flow with background velocity $\bm u_{\infty}(x,y) = \gamma (y, 0)^T, \qquad \gamma = 1$. The initial interface is a uniformly stretched circle of radius $0.5$ centered at the origin.
The total velocity is decomposed as $\bm u = \bm u_\infty + \bm u_1$, where $\bm u_1$ satisfies the interface problem~\eqref{eqn:stokes-ip} with the same jump conditions as in Example~3.

We examine two viscosity contrasts:
(i) $\mu^+ = 1$, $\mu^- = 0.5$ and
(ii) $\mu^+ = 1$, $\mu^- = 0.1$, with $\kappa^\pm = \mu^\pm$.
The time evolution of the interface for these two cases is shown in Figures~\ref{fig:eg5-case1} and~\ref{fig:eg5-case2}, respectively.
A Lagrangian marker on the interface is indicated by a green dot to track the tangential motion.

For the lower interior viscosity, case~(ii), the interface is strongly stretched by the shear flow and becomes a prolate ellipse.
For the higher interior viscosity, case~(i), the enclosed fluid behaves more like a rigid inclusion, and the deformation is weaker.
In both cases, the green marker moves tangentially along the interface, exhibiting the characteristic tank-treading motion.

\begin{figure}[htbp]
    \centering
    \includegraphics[width=0.23\linewidth]{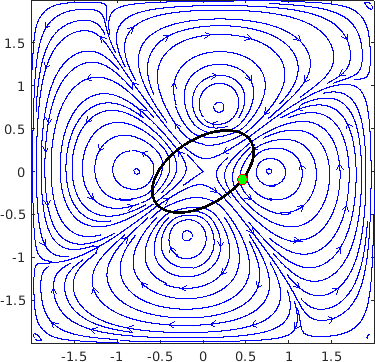}
    \includegraphics[width=0.23\linewidth]{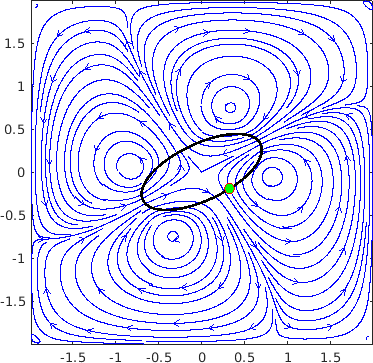}
    \includegraphics[width=0.23\linewidth]{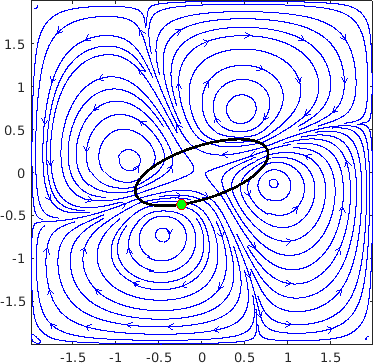}
    \includegraphics[width=0.23\linewidth]{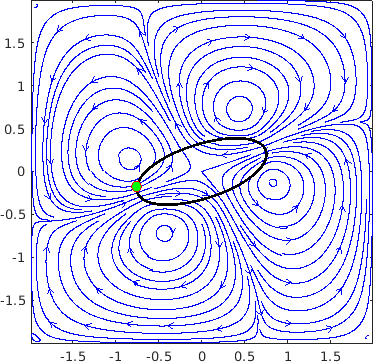}
    \caption{
    Time evolution of the elastic interface in Brinkman shear flow for 
    $\mu^+ = 1$, $\mu^- = 0.5$ with $\kappa^\pm = \mu^\pm$ at 
    $t = 1, 2, 5, 10$.
    }
    \label{fig:eg5-case1}
\end{figure}

\begin{figure}[htbp]
    \centering
    \includegraphics[width=0.23\linewidth]{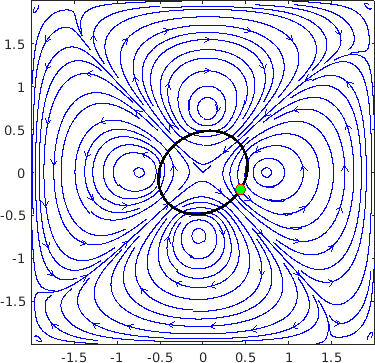}
    \includegraphics[width=0.23\linewidth]{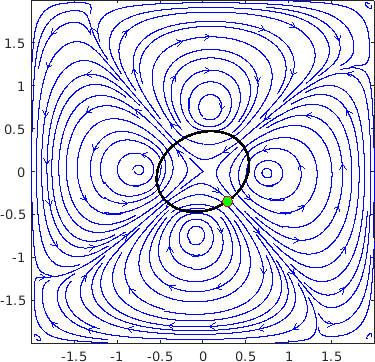}
    \includegraphics[width=0.23\linewidth]{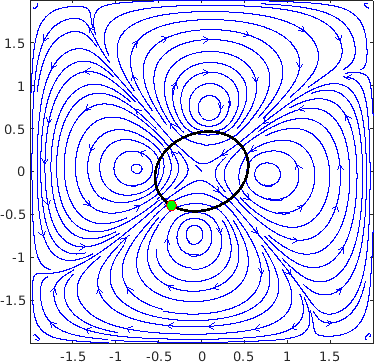}
    \includegraphics[width=0.23\linewidth]{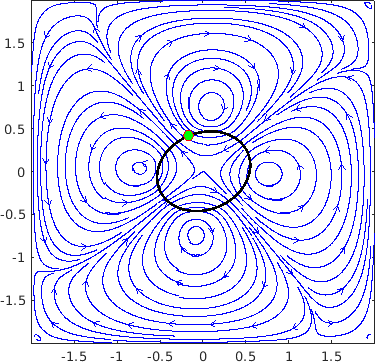}
    \caption{
    Time evolution of the elastic interface in Brinkman shear flow for 
    $\mu^+ = 1$, $\mu^- = 0.1$ with $\kappa^\pm = \mu^\pm$ at 
    $t = 1, 2, 5, 10$.
    }
    \label{fig:eg5-case2}
\end{figure}

\section{Discussion}\label{sec7}
In this work, we present a correction-function-based kernel-free boundary integral method for solving Brinkman-type interface problems in two-dimensional space. Due to the discontinuities in the coefficients and the solution, we first formulate the interface problem in boundary integral form, where the integral operators are indirectly defined as traces of normal derivatives of layer potentials that satisfy simpler constant-coefficient interface problems. Such a formulation does not require the analytical expression of the Green's function. The constant-coefficient interface problems are solved on a staggered grid using the MAC scheme with tailored correction terms at grid nodes near the interface. We introduce a collocation method to construct the correction function in a narrow band around the interface and prove the solvability of the discrete collocation problem. A significant advantage of the collocation method is that it avoids the complicated and tedious derivation of jump terms needed to correct the right-hand side and to construct piecewise smooth interpolants. It also removes the need to find the intersection points between the interface and the Cartesian grid, which is a nontrivial computational geometry task that complicates the implementation.

Several numerical examples with fixed interfaces demonstrate that the proposed method is both accurate and efficient, and it remains robust under large coefficient contrasts. 
Applications involving moving interfaces further indicate that the method is well suited for simulating complex moving boundary problems, such as those arising in fluid mechanics and biophysics.

The current methodology can be extended to three dimensions by combining it with suitable surface discretization techniques, such as triangular meshes or implicit-surface representations~\cite{Beale2006}, although this becomes substantially more intricate when the interface evolves over time. 
Future work also includes coupling the Brinkman equation with Darcy flow through more sophisticated interface conditions, such as the Beavers-Joseph-Saffman condition.




\section*{Acknowledgement}
W.~Ying is supported by the National Natural Science Foundation of China, Division of Mathematical Sciences (Project No.~12471342), and the Fundamental Research Funds for the Central Universities of China.

\bibliographystyle{elsarticle-num}
\bibliography{references}

\end{document}